\documentclass[12pt]{amsart}
\usepackage{amsmath,amsthm,amssymb}
\usepackage{graphicx}

\def\R{\mathbb{R}}
\def\C{\mathbb{C}}
\def\P{\mathbb{P}}
\def\E{\mathbb{E}}
\def\N{\mathbb{N}}

\def\CC{\mathcal{C}}

\renewcommand{\S}{\mathbb{S}}

\DeclareMathOperator{\Cov}{Cov}

\DeclareMathOperator{\polydim}{PolyDim}

\newtheorem{definition}{Definition}[section]
\newtheorem{theorem}{Theorem}[section]
\newtheorem{proposition}[theorem]{Proposition}
\newtheorem{lemma}[theorem]{Lemma}
\newtheorem{remark}{Remark}[section]
\newtheorem{corollary}[theorem]{Corollary}

\def\eps{\varepsilon}

\newcommand{\diam}{\text{Diam}}
\newcommand{\ang}{\text{Ang}}

\title{Simple Universal Bounds for Chebyshev-Type Quadratures}
\author{Ron Peled}
\address{Ron Peled\hfill\break
Courant Institute of Mathematical Sciences\\
251 Mercer st.\\
New York, NY 10012, USA \\}
\email{peled@cims.nyu.edu}
\urladdr{http://www.cims.nyu.edu/~peled}
\thanks{New York University. Partially completed during stay at the Institut Henri Poincare - Centre Emile Borel. Research supported by NSF Grant OISE 0730136.}

\subjclass[2000]{65D32; 41A55; 60D05.}
\keywords{Quadrature; Cubature; Chebyshev-type quadrature; Gaussian quadrature; Minimal quadrature size; Point sets on sphere; Poisson process.}

\begin{document}
\begin{abstract}
A Chebyshev-type quadrature for a probability measure $\sigma$ is a distribution which is uniform on $n$ points and has the same first $k$ moments as $\sigma$. We give an upper bound for the minimal $n$ required to achieve a given degree $k$, for $\sigma$ supported on an interval. In contrast to previous results of this type, our bound uses only simple properties of $\sigma$ and is applicable in wide generality.
We also obtain a lower bound for the required number of nodes which only uses estimates on the moments of $\sigma$.
Examples illustrating the sharpness of our bounds are given. 
As a corollary of our results, we obtain an apparently new result on the Gaussian quadrature.

In addition, we suggest another approach to bounding the minimal number of nodes required in a Chebyshev-type quadrature, utilizing a random choice of the nodes, and propose the challenge of analyzing its performance. A preliminary result in this direction is proved for the uniform measure on the cube. Finally, we apply our bounds to the construction of point sets on the sphere and cylinder which form local approximate Chebyshev-type quadratures. These results were needed recently in the context of understanding how well can a Poisson process approximate certain continuous distributions. The paper concludes with a list of open questions.

\end{abstract}
\maketitle



\begin{section}{Introduction}
A quadrature formula is a way of approximating a distribution by a set of point masses which preserves the integral of all polynomials up to a certain degree. More precisely, given an integer $k\ge 1$ and a measure $\sigma$ on $\R$ with finite first $k$ moments, a \emph{quadrature formula} of (algebraic) degree at least $k$ is a set of nodes $\{x_i\}_{i=1}^n\subset\R$ and weights $\{m_i\}_{i=1}^n\subset\R_+$ such that
\begin{equation}\label{quadrature_formula}
\int x^j d\sigma(x) = \sum_{i=1}^n m_i x_i^j
\end{equation}
for all integer $0\le j\le k$. The degree is exactly $k$ if equality does not hold when $j=k+1$. 
Such formulas have many applications in numerical analysis, classical analysis \cite{KN77}, geometry \cite{K98} and other fields. The maximal degree possible for a quadrature formula with $n$ nodes is $2n-1$ (unless $\sigma$ itself is atomic with $n$ nodes or less). This degree is attained uniquely for a distinguished formula called the \emph{Gaussian quadrature formula} whose $n$ nodes are placed at the roots of the $n$'th orthogonal polynomial corresponding to $\sigma$.

%

In this paper, we are concerned with a restricted class of quadrature formulas. We consider only probability measures $\sigma$ and restrict our formula to have all its weights equal (to $\frac{1}{n}$). Hence our formula takes the form
\begin{equation}\label{Chebyshev_type_quadrature}
\int x^j d\sigma(x) = \frac{1}{n}\sum_{i=1}^n x_i^j \qquad \text{for all $0\le j\le k$,}
\end{equation}
where the nodes $(x_i)_{i=1}^n$ need not be distinct. Such formulas are called \emph{Chebyshev-type quadrature formulas}. The special case when $n=k$ is known as Chebyshev quadratures, see the survey \cite{G75a}. There is also a multidimensional analogue of \eqref{Chebyshev_type_quadrature}, called \emph{Chebyshev-type cubatures}, when $\sigma$ is a measure on $\R^d$, $(x_i)_{i=1}^n\subseteq\R^d$ and we require that $\int p(x) d\sigma(x) = \frac{1}{n}\sum_{i=1}^n p(x_i)$ for all polynomials $p$ of degree at most $k$.
These formulas arise in various applications, such as combinatorics, statistics \cite{SZ84}, potential theory and geometry \cite{W93}. In addition, they recently proved essential to the understanding of fine properties of the gravitational allocation \cite{CPPR07,CPPR09}, where it was necessary to understand how well, and with what probability, can a Poisson process approximate a given continuous distribution.

The following questions arise naturally: Does a Chebyshev-type quadrature always exist for given $\sigma$ and $k$? How many nodes are required to achieve a given degree for such formulas?
\begin{definition}
For a probability measure $\sigma$ on $\R$ and integer $k\ge 1$, define $n^0_\sigma(k)$ to be the minimal number of nodes $n$ required in a Chebyshev-type quadrature \eqref{Chebyshev_type_quadrature} of algebraic degree at least $k$, or $\infty$ if no such quadrature exists. Define $n_\sigma(k)$ to be the minimal integer such that for any $n\ge n_\sigma(k)$ there exists a Chebyshev-type quadrature \eqref{Chebyshev_type_quadrature} of algebraic degree at least $k$ having exactly $n$ nodes, or $\infty$ if no such integer exists (see Theorem~\ref{CTQ_existence_theorem} below).
\end{definition}
Of course, we always have $n^0_\sigma(k)\le n_\sigma(k)\le\infty$ (see Theorem~\ref{second_example_theorem} for an example where they have different orders of magnitude).

The existence question for Chebyshev-type quadratures has been researched extensively and is well understood (see \cite{SZ84,R88,K93}). Results exist for more general formulas than \eqref{Chebyshev_type_quadrature}, involving more general spaces than $\R$ and more general functions than $x^j$. In the case of \eqref{Chebyshev_type_quadrature}, one has the following necessary and sufficient conditions.
\begin{theorem}\label{CTQ_existence_theorem}
 Given an integer $k\ge 1$ and a probability measure $\sigma$ on $\R$ with $\int |x|^kd\sigma(x)<\infty$. 
\begin{enumerate}
 \item If $\sigma$ is purely atomic with $m$ atoms and $k\ge 2m$ then the only quadrature formula \eqref{quadrature_formula} of degree at least $k$ for $\sigma$ is $\sigma$ itself. Thus, in this case, if $\sigma$ has an atom of irrational weight then $n^0_\sigma(k)=n_\sigma(k)=\infty$ and if all atoms of $\sigma$ have rational weights then $n_\sigma^0(k)<\infty$ and $n_\sigma(k)=\infty$.
\item If $\sigma$ either has a non-atomic component or it is purely atomic with $m$ atoms and $k<2m$, then $n_\sigma(k)<\infty$. Furthermore, in this case, there exists an $n_0\in\N$ such that for any $n\ge n_0$ there exists a Chebyshev-type quadrature formula \eqref{Chebyshev_type_quadrature} for $\sigma$ of degree at least $k$ having all \emph{distinct} nodes.

\end{enumerate}
\end{theorem}
Although not stated explicitly, the theorem follows readily from results of Kuijlaars \cite{K93} combined with classical results in the theory of the moment problem \cite{KN77}. We prove it in Section~\ref{existence_section}.
\begin{remark}
If the support of $\sigma$ is contained in some interval $[a,b]$ then it is sometimes desirable to have a Chebyshev-type quadrature with all nodes distinct and in the open interval $(a,b)$ (see, e.g., \cite{G75a}). It is also possible to write necessary and sufficient conditions for this case, see Remark~\ref{strict_CTQ_remark} for details.
\end{remark}


%

Theorem~\ref{CTQ_existence_theorem} does not address the quantitative question of the dependence of $n_\sigma^0(k)$ and $n_\sigma(k)$ on $\sigma$ and $k$, but part 1 of it already shows that unlike the case of the ordinary quadrature \eqref{quadrature_formula}, there is no universal upper bound on $n_\sigma^0(k)$ given only in terms of $k$. Bernstein was the first to discover the surprising fact that even for very simple $\sigma$, $n_\sigma^0$ may grow super-linearly. In two papers from 1937 \cite{B37a,B37b}, he proves the following result.
\begin{theorem}\label{Bernstein_CTQ}(Bernstein)
Let $\sigma$ be the uniform distribution on $[0,1]$. Then there exist $C,c>0$ such that for all $k\ge 1$,
\begin{equation*}
ck^2\le n^0_\sigma(k)\le Ck^2.
\end{equation*}
\end{theorem}
Aside from Bernstein's result, the asymptotic behavior of $n^0_\sigma$ (or $n_\sigma$) has been determined in only a few cases; most notably in \cite{K95}, where it was generalized to a subset of the Jacobi weight functions, and in \cite{K95-2}, where it was found for measures of the form $d\sigma=w(x)(1-x^2)^{-1/2}1_{x\in[-1,1]}dx$ for $w$ positive and analytic on $[-1,1]$. 
We mention briefly that some results exist also for Chebyshev-type cubatures.
There, research has mostly concentrated on the case where $\sigma$ is the area or volume measure of a certain set. See \cite{KM94} (and \cite{K06} for related ideas) for results on simple two and three dimensional shapes, and \cite{BV08,BV09} for recent progress on spherical designs, the case when $\sigma$ is the uniform measure on a sphere, a long standing open problem.

There also exist results in the literature: \cite{RB91} (inspired by \cite{R88}) and \cite{W91}, giving upper bounds on $n_\sigma(k)$ for general measures $\sigma$ in some class. However, these results require specific bounds on $\sigma$ which seem difficult to obtain for general measures. For example, the result of \cite{RB91} requires, as one of its ingredients, a lower bound on the smallest eigenvalue of the matrix $A=(a_{ij})_{i,j=1}^k$, where $a_{ij}:=\int (x^i-m_i)(x^j-m_j)d\sigma(x)$ and $m_i:=\int x^id\sigma(x)$. Moreover, the results require $\sigma$ to have a certain regularity: to be non-atomic with full support on some interval \cite{RB91}, or to have a density satisfying certain upper and lower bounds \cite{W91}.


This paper has several goals: First, to give an upper bound on $n_\sigma(k)$ which is given in terms of \emph{simple} properties of $\sigma$ (Theorems~\ref{one_dim_quadrature_thm}, \ref{abs_cont_upper_bound_thm} and \ref{large_atom_upper_bound_theorem}), requiring only an estimate on $\sigma$'s inverse modulus of continuity. Moreover, while the bound is particularly simple for absolutely continuous measures with bounded densities, it extends also to singular measures and even to purely atomic measures, provided some control over the size of the atoms is known.
We also give a lower bound on $n_\sigma^0(k)$ which only requires estimates on the $k-1$'st and $k$'th moments of $\sigma$ (Theorem~\ref{general_lower_bound}). Corresponding examples illustrate the sharpness of our bounds (Theorems~\ref{lower_bound_theorem} and \ref{second_example_theorem}). In particular, we find that for measures $\sigma$ supported on $[0,1]$ with essentially bounded density, $n_\sigma(k)$ may rise at most exponentially with $k$, and this rate of growth is attained for some $\sigma$.
As one corollary of our theorems, we obtain an apparently new result on the Gaussian quadrature (Corollary~\ref{Gaussian_quad_corollary}).

Second, to introduce the concept of \emph{random Chebyshev-type quadratures} (and its higher-dimensional analogues), where nodes are chosen by independent samples from $\sigma$ (Section~\ref{random_CTQ_sec}). We explain how this concept provides another way to upper bound $n^0_\sigma(k)$ and $n_\sigma(k)$ and propose the challenge of analyzing its performance. A preliminary result in this direction is proven, for the case that $\sigma$ is the uniform measure on a cube (Theorem~\ref{positive_density_thm}). Our analysis proceeds via a local limit theorem.

Third, to describe applications of our theorems to the construction of point sets on spheres and cylinders which are local approximate Chebyshev-type cubatures, meaning that one may partition the sphere or cylinder to small diameter sets on which the point sets are approximate Chebyshev-type cubatures (Theorems~\ref{special_sphere_cubature} and \ref{special_cubature_on_cyl_thm}). These constructions and the result for the uniform measure on the cube mentioned above, were needed recently in the study of the gravitational allocation \cite{CPPR09} in the context of understanding how well a Poisson process can approximate certain continuous distributions.

These goals are developed in the next three subsections, without proofs. Section~\ref{proofs_section} contains proofs and supplements. Section~\ref{open_questions_section} presents open questions.
 
\subsection{Simple bounds for the number of nodes}\label{universal_bound_section}
In this section, we present an upper bound on $n_\sigma(k)$ which is calculated in terms of \emph{simple} properties of $\sigma$. We also give a lower bound on $n_\sigma^0(k)$ which only requires estimates on the $k-1$'st and $k$'th moment of $\sigma$, and give examples illustrating the sharpness of our bounds.
The information about $\sigma$ we shall need for our upper bound is contained in the following function,
\begin{equation}\label{R_sigma_def}
R_\sigma(\delta):=\min\left(|x-y|\ \big|\ x,y\in\R,\ \sigma([x,y])\ge \delta\right)
\end{equation}
defined for $0<\delta<1$. $R_\sigma$ is the inverse modulus of continuity of $\sigma$; $R_\sigma(\delta)$ measures the minimal length an interval needs to have in order to have probability at least $\delta$.
\begin{theorem}\label{one_dim_quadrature_thm}
Let $\sigma$ be a probability measure with $\sigma([0,1])=1$. Fix an integer $k\ge 2$ and let
\begin{equation*}
\begin{split}
\rho&:=(k-1)R_\sigma\left(\frac{1}{k+3}\right),\\
r&:=\frac{\rho}{6(k+3)}\left(\frac{\rho}{12e}\right)^{k-1}.
\end{split}
\end{equation*}
Then for each integer $n\ge r^{-1}$ and each $p\in\R^k$ satisfying 
\begin{equation}\label{p_condition}
\left|p_j-\int x^jd\sigma(x)\right|\le r\qquad 1\le j\le k
\end{equation}
there exist (not necessarily distinct) $(x_i)_{i=1}^n\subseteq [0,1]$ satisfying
\begin{equation*}
\frac{1}{n} \sum_{i=1}^n x_i^j = p_j \quad\text{ for all integer $1\le j\le k$.}
\end{equation*}
\end{theorem}
The theorem states that if the number of nodes $n$ is large enough with respect to $k$ and the quantity $R_\sigma\left(\frac{1}{k+3}\right)^{-1}$, then there exists a Chebyshev-type quadrature \eqref{Chebyshev_type_quadrature} having the same first $k$ moments as $\sigma$. Moreover, for each small perturbation of the moments of $\sigma$, there exists a Chebyshev-type quadrature with these perturbed moments. The theorem gives explicit bounds on $n$ and on the size of the allowed perturbation. Note that to have a non-trivial bound, we must have $R_\sigma(\frac{1}{k+3})>0$, which is equivalent to saying that $\sigma$ has no atom with mass at least $\frac{1}{k+3}$.
For generalizations of the theorem to the case of distinct nodes in $(0,1)$ and to the case of functions other than $x^j$, see Remark~\ref{main_theorem_extensions_remark}.

Of course, the most important case of the theorem is when the moments of $\sigma$ are unperturbed. In addition, in many applications, one is interested in absolutely continuous distributions with bounded densities. If the density bound is $M$, we have $R_\sigma(\delta)\ge M^{-1}\delta$ for all $\delta$. Similarly, if one considers singular $\sigma$, a typical scenario is when $R_\sigma(\delta)\ge c\delta^{\beta}$ for some $\beta> 1$. For these cases we have the following useful corollary which follows immediately from Theorem~\ref{one_dim_quadrature_thm}.
\begin{theorem}\label{abs_cont_upper_bound_thm}
Let $\sigma$ be a probability measure with $\sigma([0,1])=1$. 
\begin{enumerate}
 \item Suppose that $\sigma$ is absolutely continuous with a density which is essentially bounded by $M$. Then for each integer $k\ge 2$ we have 
\begin{equation*}
 n_\sigma(k)\le \lceil 75e^4 kM\left(12eM\right)^{(k-1)} \rceil.
\end{equation*}
\item Suppose that $R_\sigma(\delta)\ge c\delta^{\beta}$ for some $c>0, \beta\ge 1$ and all $0<\delta<1$. Then for each integer $k\ge 2$ we have
\begin{equation*}
 n_\sigma(k)\le \lceil \alpha^k(k+3)^{(\beta-1)k+1}\rceil,
\end{equation*}
where $\alpha>0$ depends only on $c$ and $\beta$.
\end{enumerate}
Furthermore, in both cases we have that all quadrature nodes lie in $[0,1]$.
\end{theorem}
Hence, for measures with bounded densities, one needs at most an exponential number of nodes in a Chebyshev-type quadrature. A more singular measure may require even more nodes. Similar conclusions can be drawn for other measures according to which lower bound one has for $R_\sigma$. 

The previous theorems provide quantitative bounds for $n_\sigma(k)$ in the cases when $\sigma$ does not have large atoms. Can we provide similar bounds when $\sigma$ is a mixture of a large atom and a non-atomic component? or when $\sigma$ has infinitely many atoms? The following theorem does so. Define, for a probability measure $\sigma$ and $0< \eps<1$,
\begin{equation*}
\begin{split}
\sigma^t_\eps&:=\sigma - \sum_{\{x\ |\ \sigma(\{x\})>\eps\}} (\sigma(\{x\})-\eps)\delta_x\ ,\\ \sigma'_\eps&:=\frac{\sigma^t_\eps}{\sigma^t_\eps(\R)}.
\end{split}
\end{equation*}
In words, $\sigma^t_\eps$ is $\sigma$ with all its atoms truncated to mass $\eps$ and $\sigma'_\eps$ is its normalized version.
\begin{theorem}\label{large_atom_upper_bound_theorem}
Let $\sigma$ be a probability measure with $\sigma([0,1])=1$. Fix an integer $k\ge 2$ and suppose that there exists $0<\eps<1$ such that
\begin{equation}\label{eps_condition}
\frac{\eps}{\sigma^t_\eps([0,1])}<\frac{2}{2k+7}.
\end{equation}
Fix such an $\eps$ and let
\begin{equation*}
\begin{split}
\rho&:=(k-1)R_{\sigma'_{\eps}}\left(\frac{2}{2k+7}\right),\\
r&:=\frac{\rho}{6(k+3)}\left(\frac{\rho}{12e}\right)^{k-1}.
\end{split}
\end{equation*}
Then for any integer $n\ge \max\left(\frac{1}{r\sigma^t_\eps([0,1])},\frac{2k+6}{\eps}\right)$ there exist (not necessarily distinct) $(x_i)_{i=1}^n\subseteq [0,1]$ satisfying
\begin{equation*}
\frac{1}{n} \sum_{i=1}^n x_i^j = \int x^jd\sigma(x) \quad\text{ for all integer $1\le j\le k$.}
\end{equation*}
\end{theorem}
\begin{remark}\label{atomic_theorem_remark}
 \begin{enumerate}
\item It is not difficult to see that condition $\eqref{eps_condition}$ is satisfied for any small enough $\eps>0$ if $\sigma$ has a non-atomic component or at least $k+4$ atoms.
  \item Note that the largest atom in $\sigma'_{\eps}$ has at most $\frac{\eps}{\sigma^t_{\eps}([0,1])}$ mass so that condition \eqref{eps_condition} ensures that $\rho>0$. The reason that the $\frac{1}{k+3}$ of Theorem~\ref{one_dim_quadrature_thm} is replaced by $\frac{2}{2k+7}$ and for the extra factor $\frac{2k+6}{\eps}$ in the bound on $n$ is that we may not be able to exactly truncate the atoms of $\sigma$ to probability $\eps$ using atoms of size $\frac{1}{n}$.
\item Similarly to Theorem~\ref{one_dim_quadrature_thm}, we can quantify a statement saying that for any moment vector $p$ which is close enough to the moments of $\sigma$, we can find $\{x_i\}_{i=1}^n$ with these moments.
\item The proof is based on writing $\sigma=q\sigma_1+(1-q)\sigma_2$ for probability measures $\sigma_1,\sigma_2$, where $\sigma_1$ approximates $\sigma'_{\eps}$ and $\sigma_2$ is the ``leftovers'' of the large atoms of $\sigma$. 
The approximation is chosen so that $\sigma_2$ already has atoms with rational probability, 
then Theorem~\ref{one_dim_quadrature_thm} is used to get a Chebyshev-type quadrature for $\sigma_1$. We note that this approach might yield better bounds than those of Theorem~\ref{one_dim_quadrature_thm} even for $\sigma$ which do not have large atoms. For example, if $R_\sigma(\frac{1}{k+3})$ is very small, one may try to decompose $\sigma=q\sigma_1+(1-q)\sigma_2$ so that $R_{\sigma_1}(\frac{1}{k+3})>R_{\sigma}(\frac{1}{k+3})$ and $q$ is rational with small denominator. Then approximate $\sigma_2$ in a simple manner, say as in Lemma~\ref{simple_approximation_lemma} below, and finally approximate $\sigma_1$ using Theorem~\ref{one_dim_quadrature_thm} and use the freedom in the moments afforded by \eqref{p_condition} to compensate for the errors in the moments of the approximation to $\sigma_2$.
 \end{enumerate}
\end{remark}

\subsubsection*{Lower bounds}
In this section we complement the above upper bounds for $n_\sigma(k)$ by presenting lower bounds for $n_\sigma^0(k)$ and examples illustrating the sharpness of our bounds. As a by-product of our results, we note an apparently new inequality for the Gaussian quadrature.

We start by describing a lower bound for $n_\sigma^0(k)$, for general probability measures $\sigma$, which Bernstein used in deriving Theorem~\ref{Bernstein_CTQ}.
To state it, we first recall that for $k=2m-1$ for $m\in\N$ and a probability measure $\sigma$ on $\R$ with $\int |x|^kd\sigma(x)<\infty$, unless $\sigma$ is purely atomic with less than $m$ atoms, we have the Gaussian quadrature formula with nodes $\xi_1^{(m)}<\xi_2^{(m)}<\cdots<\xi_m^{(m)}$ and weights $(\lambda_i^{(m)})_{i=1}^m$ satisfying
\begin{equation}\label{Gaussian_quadrature}
 \sum_{i=1}^m \lambda_i^{(m)}(\xi_i^{(m)})^j = \int x^jd\sigma(x)\qquad \forall\ 0\le j\le k.
\end{equation}
\begin{theorem}(Bernstein \cite{B37a})\label{Bernstein_estimate}
For a probability measure $\sigma$ and $k=2m-1$ as above, we have $n_\sigma^0(k)\ge \frac{1}{\min(\lambda_1^{(m)}, \lambda_m^{(m)})}$.
\end{theorem}
Bernstein proved this theorem in the special case of the uniform distribution on an interval, but as some authors note \cite{G75b,K93,KM94}, the bound extends to all measures. We note an immediate corollary of Theorem~\ref{abs_cont_upper_bound_thm} and Theorem~\ref{Bernstein_estimate} to an estimate on Gaussian quadratures.
\begin{corollary}\label{Gaussian_quad_corollary}
For any probability measure $\sigma$ with $\sigma([0,1])=1$ and density essentially bounded by $M$, we have for any $m\in\N$ that
\begin{equation*}
\lambda_1^{(m)}\ge\frac{1}{\lceil 75e^4 (2m-1)M\left(12eM\right)^{(2m-2)} \rceil}.
\end{equation*}
\end{corollary}
This estimate appears to be new and we do not know if it is simple to prove directly. Similar corollaries can be phrased for general measures using $R_\sigma$ and Theorem~\ref{one_dim_quadrature_thm}.

Theorem~\ref{Bernstein_estimate} can be quite accurate (as Theorem~\ref{Bernstein_CTQ} illustrates), however, one drawback of it is that it may be difficult to apply in specific cases since it requires knowledge of the Gaussian quadrature associated to the given measure.
We now propose a second lower bound, whose proof makes use of similar ideas to that of Theorem~\ref{Bernstein_estimate}, which has the advantage that in order to apply it, the only required information about the measure are bounds on its $(k-1)$'st and $k$'th moments.
\begin{theorem}\label{general_lower_bound}
Let $\sigma$ be a probability measure on $\R$ with $\sigma(\{0\})<1$. Then for every \emph{odd} integer $k\ge 3$ for which $\int |x|^k d\sigma(x)<\infty$, we have
\begin{equation}\label{n_sigma_0_lower_bound}
n_\sigma^0(k)\ge\frac{\left(\int x^kd\sigma(x)\right)^{k-1}}{\left(\int x^{k-1}d\sigma(x)\right)^{k}}.
\end{equation}
\end{theorem}
We remark that the lower bound given by the above theorem changes, in general, when replacing $\sigma$ by a translate of it. Hence, one may wish to optimize the amount by which to translate $\sigma$ before applying the bound. To keep the theorem as simple as possible, we avoid making this optimization here.

Theorem~\ref{general_lower_bound} implies that, for example, to obtain a lower bound on $n_\sigma^0(k)$ for probability measures supported on $[0,1]$ (which are not $\delta_0$), it is sufficient to have a lower bound on the absolute value of the $k$'th moment of $\sigma$ and an upper bound on the $(k-1)$'st moment of $\sigma$. The following corollary is proved via this technique. It illustrates that the upper bound given by Theorem~\ref{abs_cont_upper_bound_thm} and the lower bound given by Theorem~\ref{general_lower_bound} may be close in specific examples.
\begin{corollary}\label{lower_bound_theorem}
 There exists $C>0$ such that for every odd integer $k>C$, there exists a probability measure $\sigma_k$ on $[0,1]$, absolutely continuous with density essentially bounded above by $Ck$, satisfying
\begin{equation*}
 n^0_{\sigma_k}(k)\ge \frac{1}{2\sqrt{k}}\left(\frac{e}{2}\right)^{k}.
\end{equation*}
\end{corollary}
In fact, the measure $\sigma_k$ constructed in this corollary is simply the exponential distribution, properly truncated and rescaled. Note also that, since $n_\sigma^0(k)$ is non-decreasing in $k$ for any measure $\sigma$, the corollary implies a similar bound for even integers $k$.

Let us compare the lower bound of Corollary~\ref{lower_bound_theorem} with the upper bound on $n_\sigma(k)$ given by Theorem~\ref{abs_cont_upper_bound_thm}. Since the density of $\sigma_k$ is bounded by $Ck$ for some $C>0$, Theorem~\ref{abs_cont_upper_bound_thm} gives
\begin{equation*}
n_{\sigma_k}^0(k)\le n_{\sigma_k}(k)\le (C'k)^k
\end{equation*}
for some $C'>0$, which differs from the bound of Corollary~\ref{lower_bound_theorem} by a $\log k$ factor in the exponent. Seeking to have an example on which the upper bound of Theorem~\ref{abs_cont_upper_bound_thm} is sharp, up to the constants involved, we introduce the following second example.
We set $d_n(\sigma)$, for $n\in\N$ and a probability measure $\sigma$ on $\R$ with all moments finite, to be the maximal possible degree of accuracy for a Chebyshev-type quadrature for $\sigma$ having exactly $n$ (not necessarily distinct) nodes (or $\infty$ if any degree of accuracy can be attained).

\begin{theorem}\label{second_example_theorem}
Let $\sigma_0$ be the probability measure having density
\begin{equation*}
w(x):=\begin{cases}
1& x\in [-1,-\frac{1}{2}]\cup[\frac{1}{2},1]\\
0& \text{otherwise}
\end{cases}.
\end{equation*}
In other words, $\sigma_0$ is the uniform distribution on the set $[-1,-\frac{1}{2}]\cup[\frac{1}{2},1]$. Then there exist $C,c>0$ such that 
\begin{align}
&d_n(\sigma_0)\ge c\sqrt{n}\qquad\qquad\text{for even $n$,}\label{even_n_bound}\\
&d_n(\sigma_0)\le C\ln(Cn)\qquad\text{for odd $n$.}\label{odd_n_bound}
\end{align}
In particular, there exist $C_1,c_1>0$ such that for every $k\in\N$, we have
\begin{align}
&n_{\sigma_0}^0(k)\le C_1k^2\qquad\text{and}\label{example_n_sigma_0_bound}\\
&n_{\sigma_0}(k)\ge c_1e^{c_1k}.\label{example_n_sigma_bound}
\end{align}
\end{theorem}
Thus the theorem shows the surprising fact that for the uniform distribution on two disjoint intervals $\sigma_0$, $d_n(\sigma_0)$ has completely different orders of magnitude for odd and even $n$. It also shows that the upper bound of Theorem~\ref{abs_cont_upper_bound_thm} can be attained, up to the constants involved.
Our proof of this theorem uses general theorems of Peherstorfer \cite{P90} and F\"orster \& Ostermeyer \cite{FO86} which, when taken together, show that $d_n(\sigma)$ may rise at most logarithmically in $n$ for odd $n$, whenever $\sigma$ is a symmetric measure having $0$ outside its support. We remark also that the phenomena that $d_n(\sigma)$ may have very different orders of magnitude for odd and even $n$, was first discovered, in a particular case, by F\"orster \cite{F86}.

\subsection{Random Chebyshev-type cubatures}\label{random_CTQ_sec}
We give the name \emph{random} \linebreak Chebyshev-type cubature to the situation in which we would like to approximate the moments of a measure $\sigma$ by the moments of a uniform distribution on $n$ points (as in ordinary Chebyshev-type cubatures), but do not choose the position of the points, instead, the points are chosen randomly according to independent samples from $\sigma$. In such a situation, it is natural to ask how small is the probability that the moments of the random measure approximate the moments of $\sigma$ very well. In general, this is a question about a \emph{small ball probability}. 

As we shall see, this notion gives another way to prove existence of Chebyshev-type quadratures and cubatures and we believe that it deserves better study. In addition, in the analysis of \cite{CPPR09}, it arose naturally in the context of understanding how well a Poisson process approximates Lebesgue measure. 



To formalize the above, fix $d\ge 1$ and define for $k\ge 1$, $\polydim(k,d):=\binom{k+d}{d}-1$. Then define the moment map $P_k^d:\R^d\to\R^{\polydim(k,d)}$ by
\begin{equation}\label{moment_map_def}
P_k^d(x):=(x^{\alpha})_\alpha
\end{equation}
where $\alpha$ runs over all multi-indices with $0<|\alpha|\le k$, where we mean that $\alpha\in(\N\cup \{0\})^d$, $x^\alpha:=\prod_{i=1}^d x_i^{\alpha_i}$ and $|\alpha|:=\sum_{i=1}^d\alpha_i$.

Given a probability measure $\sigma$ with support in $\R^d$, let $M_k(\sigma)$ denote its vector of multi-moments of degree at most $k$. That is,
\begin{equation*}
 M_k(\sigma):=\left(\int x^\alpha d\sigma(x)\right)_{0<|\alpha|\le k}.
\end{equation*}
For $n\ge 1$, consider the following random measure
\begin{equation*}
 \sigma_n:=\frac{1}{n}\sum_{i=1}^n \delta_{x_i}
\end{equation*}
where the $\{x_i\}_{i=1}^n$ are chosen independently from the distribution $\sigma$. Note that $M_k(\sigma_n)=\frac{1}{n}\sum_{i=1}^n P_k^d(x_i)$. It follows that $\E M_k(\sigma_n)=M_k(\sigma)$. Still, the moments of $\sigma_n$ typically do not approximate well the moments of $\sigma$. Indeed, by the central limit theorem, the difference $|M_k(\sigma_n)_\alpha - M_k(\sigma)_\alpha|$ scales like $\frac{1}{\sqrt{n}}$ for any fixed $\alpha$ (if $\sigma$ has moments of any order, say). We are interested in the probability that this difference is much smaller. More precisely, let
\begin{equation*}
 p_{n,k,\eps}(\sigma):=\P\left(\Vert M_k(\sigma_n) - M_k(\sigma)\Vert_\infty\le \frac{\eps}{\sqrt{n}}\right).
\end{equation*}
This is the \emph{small ball probability} for the random vector $M_k(\sigma_n)$.
We would like to understand how it scales for a fixed $n$ as $\eps$ tends to $0$. The following lemma connects this probability to the existence of Chebyshev-type cubatures.
\begin{lemma}\label{small_ball_to_cubature_lemma}
If for some $n,k\ge 1$ and every $\eps>0$ we have $p_{n,k,\eps}(\sigma)>0$ then there exists a Chebyshev-type cubature for $\sigma$ of degree at least $k$ having exactly $n$ (not necessarily distinct) nodes.
\end{lemma}
Thus understanding $p_{n,k,\eps}$ provides a different way to show existence of \linebreak Chebyshev-type cubatures (and to prove lower bounds for $n_\sigma^0$). We propose the challenge of bounding, in specific examples, the minimal $n$ for which the condition of the lemma is satisfied and seeing if this approach may improve known bounds. The most interesting case in this respect is that of spherical designs, when $\sigma$ is the uniform measure on the sphere $\S^{d-1}$, but one may start by checking what bound is achieved for the interval and comparing it with Theorem~\ref{Bernstein_CTQ}. In this paper, we content ourselves with a small step in this direction (which, however, already takes some work to prove) by showing that the condition of the lemma is satisfied, for large enough $n$, when $\sigma$ is the uniform measure on the cube $[-1,1]^d$. This is achieved by showing that, for large enough $n$, the random vector $M_k(\sigma_n)$ has a positive density at $M_k(\sigma)$. Unfortunately, our result does not provide quantitative bounds for $n$.

\begin{theorem} \label{positive_density_thm}
Fix $k\ge 1$ and let $(X_i)_{i=1}^\infty$ be an IID sequence of RV's uniform on $[-1,1]^d$. Let $M_i:=P_k^d(X_i)$ and $\bar{S}_n:=\frac{1}{\sqrt{n}}\sum_{i=1}^n (M_i-\E M_1)$. Then there exists $N_0=N_0(k,d)>0$, $a=a(k,d)>0$ and $t=t(k,d)>0$ such that for all $n>N_0$, $\bar{S}_n$ is absolutely continuous with respect to Lebesgue measure in $\R^{\polydim(k,d)}$ and its density $f_n(x)$ satisfies $f_n(x)\ge a$ for $|x|\le t$.
\end{theorem}
This theorem can be viewed as a local limit theorem (and its proof follows this approach). The moment vector $\sqrt{n}\bar{S}_n$ is a sum of IID contributions and we show that starting at some large $n$, it has density which (suitably scaled) converges uniformly to the density of a centered Gaussian vector. Our main tool is Fourier analytic estimates.

\subsection{Application to construction of local cubatures}\label{local_CTQ_sec}
In this section, we apply the results of Section~\ref{universal_bound_section} to give a construction of discrete point sets on the sphere and the cylinder which are approximate Chebyshev-type cubatures.
In the case of the sphere, we approximate its surface measure. In the case of the cylinder, the approximation is to a measure with density (with respect to surface area) constant on every spherical section and growing linearly along the axis of the cylinder. In both cases, our approximations are stronger than ordinary Chebyshev-type cubatures in that they are ``local'', i.e., there is a partition of the set in question (the sphere or the mid-part of the cylinder) to subsets of small diameter such that our point set restricted to each of these subsets is an approximate Chebyshev-type cubature.

The application to the cylinder, which builds on the application to the sphere, was central in the recent study \cite{CPPR09} on gravitational allocation where it was used to construct ``wormholes''; long tentacles in space surrounded by rings of stars in which the gravitational force is atypically strong in the tentacle's direction. 

Our construction of the Chebyshev-type cubature for the sphere is very similar to a construction of Wagner \cite{W92} which he used in his work on a problem in potential theory (in a somewhat similar manner as the application in \cite{CPPR09}). Despite this similarity, we chose to give a full proof of it here since some parts in Wagner's construction (such as the exact partition of the sphere) are only sketched and since our construction 
gives explicit bounds on the number of nodes in the cubature (at the expense of getting only an approximate cubature formula), whereas his only shows existence.

To state our theorems, define $\sigma_d$ to be the $d$-dimensional area measure on sets in $\R^d$, and, abusing notation slightly, also as the $d$-dimensional Hausdorff measure on sets in $R^{d'}$ for $d'>d$.
For a set $E\subseteq \R^d$, define $\diam(E):=\max_{x,y\in E}|x-y|$, the diameter of $E$, where $|\cdot|$ is Euclidean distance. We use the same multi-index notation as defined after \eqref{moment_map_def}. Define for $d,k\ge 1$ and $\delta>0$,
\begin{equation}\label{m_0_definition}
m_0(d,k,\delta):=\text{Smallest integer $m\ge 1$ satisfying }\left(\frac{ke}{m+1}\right)^{m+1}\le \frac{\delta}{2d2^k}.
\end{equation}
Embedding $\S^d$ into $\R^{d+1}$ as the unit sphere, we prove:
\begin{theorem}\label{special_sphere_cubature}
For each $d\ge 1, k\ge 1$ and $0<\tau,\delta<1$ there exist $C(d),c(d)>0$ (depending only on $d$), an integer $K=K(\tau,d)>0$ and a partition of $\S^{d}$ (up to surface measure $0$) into measurable subsets $E_1, \ldots, E_K$ satisfying the following properties:
\begin{enumerate}
\item $\diam(E_i)\le C(d)\tau$, $\sigma_{d}(E_i)\ge c(d)\tau^{d}$ for all $i$, and $K\le C(d)\tau^{-d}$.
\item For $N=n^d$, where $n$ can be any integer satisfying $n\ge C(d)^{m_0(d,k,\delta)}$, and for each $1\le i\le K$, there exist $(z_{i,j})_{j=1}^N\subseteq E_i$ such that
\begin{equation}\label{approx_quadrature_on_sphere}
\left|\frac{1}{N}\sum_{j=1}^N g(z_{i,j}) - \frac{1}{\sigma_{d}(E_i)}\int_{E_i} g(z)d\sigma_{d}(z)\right|\le \delta
\end{equation}
for all $g:\R^{d+1}\to\R$ of the form $g(z)=(z-w)^\alpha$ for $w\in\S^d$ and a multi-index $\alpha$ with $|\alpha|\le k$.
\end{enumerate}
\end{theorem}

To state our theorem for the cylinder, we make a few more definitions.
Given $L,W>0$ and a dimension $d\ge 1$, let 
\begin{equation*}
 P_{L,W}:=\{x\in\R^d\ |\ |x_1|\le L, x_2^2+\cdots +x_d^2=W^2\}.
\end{equation*}
so that $P_{L,W}$ is the curved part of the boundary of a length $L$ cylinder of radius $W$.
Let $\nu_{L,W}$ be the measure supported on $P_{L,W}$ and absolutely continuous with respect to $\sigma_{d-1}$ with density $V(x_1,\ldots, x_d)=v(x_1)=1+\frac{x_1+L}{2L}$. I.e., the density increases linearly from $1$ to $2$ as $x_1$ increases from $-L$ to $L$.
Recalling the definition of $m_0(d,k,\delta)$ from \eqref{m_0_definition}, we prove:
\begin{theorem}\label{special_cubature_on_cyl_thm}
For each $d\ge 3$ there exists $C(d)>0$ such that for each $k\ge 1$, $L>C(d)$, $W>0$, $0<\tau<W$ and $0<\delta<W^k$, we have an integer $K=K(L,W,\tau,d)>0$ and measurable subsets $D_1, \ldots, D_{K}\subseteq P_{2L,W}$ satisfying the following properties:
\begin{enumerate}
\item[(I)] $\nu_{2L,W}(D_i\cap D_j)=0$ for each $i\neq j$ and $\nu_{2L,W}\left(P_{L,W}\setminus\left(\cup_{i=1}^{K} D_i\right)\right)=0$.
\item[(II)] $\diam(D_i)\le C(d)\tau$, $\nu_{2L,W}(D_i)= \tau^{d-1}$ for all $i$, and $K\le C(d)LW^{d-2}\tau^{-(d-1)}$.
\item[(III)] For $n=n_1^{d-1}$, where $n_1$ can be any integer satisfying $n_1\ge C(d)^{m_0(d-2,k,\delta/(4LW)^k)}$, and for each $1\le i\le K$, there exist $(w_{D_{i},j})_{j=1}^n\subseteq D_i$ such that
\begin{equation*}
\left|\frac{1}{n}\sum_{j=1}^n h(w_{D_i,j}) - \frac{1}{\nu_{2L,W}(D_i)}\int_{D_i} h(w)d\nu_{2L,W}(w)\right|\le \delta
\end{equation*}
for all $h:\R^{d}\to\R$ of the form $h(w)=(w-y)^\alpha$ for $y\in P_{2L,W}$ and a multi-index $\alpha$ with $|\alpha|\le k$.
\end{enumerate}
\end{theorem}
It is worth noting that in the above theorem, the sets $D_1,\ldots, D_K$ cover $P_{L,W}$ (up to $\nu$-measure 0) and are contained in $P_{2L,W}$, but do not necessarily form a partition of $P_{2L,W}$ (up to $\nu$-measure 0). Indeed, this is not possible for generic values of $L,W$ and $\tau$ since property (I) implies the $D_i$ are disjoint (up to $\nu$-measure 0) and property (II) implies that each $D_i$ has $\nu$ measure exactly $\tau^{d-1}$. Thus, for the sets to form a partition, we would need that $\frac{\nu_{2L,W}(P_{2L,W})}{\tau^{d-1}}$ be an integer.
\begin{subsection}{Acknowledgments}
We are grateful to Franz Peherstorfer for explaining to us the relevance of \cite{P90} and suggesting that it may be used to obtain lower bounds for $n_\sigma(k)$, as we did in Theorem~\ref{second_example_theorem}. We thank Nir Lev for his essential help in referring us to the book of Stein and explaining the relevance of its propositions to estimates in Section~\ref{random_cubatures_on_cube_section}.
 We thank Boris Tsirelson and Mikhail Sodin for several useful conversations, in particular concerning approximation of continuous measures by discrete ones. Finally, we thank Greg Kuperberg, Yuval Peres, Dan Romik and Sasha Sodin for useful discussions and comments on quadratures and cubatures.
\end{subsection}
\end{section}

\begin{section}{Proofs and supplements}\label{proofs_section}
\begin{subsection}{Existence of Chebyshev-type quadratures}\label{existence_section}

In this section we prove Theorem~\ref{CTQ_existence_theorem}. We start with the following theorem which follows from classical results in the theory of the moment problem:



\begin{theorem}\label{quadrature_exists_thm}
 Given $k=2m-1$ for $m\in\N$ and a probability measure $\sigma$ on $\R$ with $\int |x|^kd\sigma(x)<\infty$. Unless $\sigma$ is purely atomic with less than $m$ atoms, there exists a quadrature formula \eqref{quadrature_formula} for $\sigma$ of degree at least $k$ having exactly $2m+1$ nodes.
\end{theorem}
\begin{proof}
We assume without loss of generality that $\sigma$ is atomic with $m$ nodes since otherwise we can replace $\sigma$ by its Gaussian quadrature \eqref{Gaussian_quadrature}.
Then its support is contained in an interval $[a,b]$. Fix $c<a$ and $d>b$. Note that for any polynomial $P\not\equiv0$ of degree at most $k$ which is non-negative on $[c,d]$ we have $\int_c^d Pd\sigma>0$. In other words (see \cite[III \textsection 1]{KN77}), $\sigma$ (or rather its first $k$ moments) is strictly positive with respect to $[c,d]$ and $k$. 
This implies that there exists a quadrature formula $\sigma_{m+1}$ for $\sigma$ having degree at least $k$, exactly $m+1$ nodes, all in $(c,d)$ and all different from those of $\sigma$ (this is any of the lower representations of index $n+3$, see \cite[III \textsection7.1]{KN77}. All nodes are in the interior of $[c,d]$ since $k$ is odd). Then the measure $\frac{1}{2}(\sigma+\sigma_{m+1})$ satisfies the requirements of the theorem.\qedhere


\end{proof}
We also need two theorems of Kuijlaars:
\begin{theorem}\label{Kuijlaars_distinct_nodes}(\cite{K93}, Theorem 3.2) Given a probability measure $\sigma$, suppose we have a Chebyshev-type quadrature formula \eqref{Chebyshev_type_quadrature} of degree at least $k$ with $n$ nodes in $(-1,1)$, of which $n_0\ge k$ are distinct. Then there exists a Chebyshev-type quadrature formula with $n$ distinct nodes in $(-1,1)$ of degree at least $k$. 
\end{theorem}
\begin{theorem}\label{Kuijlaars_open_set}(\cite{K93}, Theorem 4.2) Given a probability measure $\sigma$ and a quadrature formula \eqref{quadrature_formula} with weights $(m_i)_{i=1}^n$, distinct nodes in $(-1,1)$ and degree at least $n-1$, there exists a relatively open subset $U$ of the collection $\{(p_1,\ldots,p_n)\ |\ p_i>0\text{ for $1\le i\le n$ and }\sum_{i=1}^n p_i=1\}$ with $(m_i)_{i=1}^n\in U$ and such that for every $(p_i)_{i=1}^n\in U$ there exist nodes $(\tilde{x}_i)_{i=1}^n\subseteq(-1,1)$ satisfying
\begin{equation*}
 \sum_{i=1}^n p_i\tilde{x_i}^j=\int x^jd\sigma(x)\qquad 0\le j\le n-1.
\end{equation*}
\end{theorem}
The second theorem was proven in \cite{K93} for absolutely continuous $\sigma$ with bounded support but the (short) proof is valid for any $\sigma$.

\begin{proof}[Proof of Theorem~\ref{CTQ_existence_theorem}]
If $\sigma$ is purely atomic with $j\le \frac{k}{2}$ atoms, it is well known that $\sigma$ itself is the only quadrature formula having degree at least $k$. This can be seen by considering the non-negative polynomial $P$ having a double zero at each atom of $\sigma$. Since $P$ has degree $2j\le k$ and since $\int Pd\sigma=0$ we see that the integral of $P$ is zero also with respect to a quadrature with degree at least $k$. Hence, the nodes of that quadrature are a subset of the nodes of $\sigma$, but this implies that they are equal since the location of the nodes determines the weights by solving a linear system with a Vandermonde coefficient matrix.

Assume now that $\sigma$ is not purely atomic with $j\le\frac{k}{2}$ atoms. We may assume $k=2m-1$ for some $m\in\N$ since if $k$ is even, the theorem remains true when $k$ is replaced by $k+1$. We use Theorem~\ref{quadrature_exists_thm} to obtain $\sigma_{2m+1}$, a quadrature for $\sigma$ of degree at least $k$ having exactly $2m+1$ nodes. It follows from Theorem~\ref{Kuijlaars_open_set} that for some $n_0\in\N$ and any $n\ge n_0$ there exist Chebyshev-type quadratures with $n$ nodes having the same first $2m$ moments as $\sigma_{2m+1}$, so that, in particular, they have degree at least $k$ with respect to $\sigma$. The nodes of these quadratures can be made distinct using Theorem~\ref{Kuijlaars_distinct_nodes}.
\end{proof}
\begin{remark}\label{strict_CTQ_remark}
For a probability measure $\sigma$ with support in $[a,b]$ and $k\in\N$, we say that $\sigma$ is singular with respect to $[a,b]$ and $k$ if there exists a polynomial $P\not\equiv 0$ of degree at most $k$ which is non-negative on $[a,b]$ and such that $\int Pd\sigma=0$. Equivalently, $\sigma$ is singular with respect to $[a,b]$ and $k$ if and only if it is purely atomic and its index $I(\sigma)\le k$ where $I(\sigma):=\sum I(x)$ over all atoms $x$ of $\sigma$ and
\begin{equation*}
 I(x)=\begin{cases} 1& x=a\text{ or } x=b\\2& x\in(a,b)\end{cases}.
\end{equation*} 
If $\sigma$ is singular, it follows from the same proof as above, but using the polynomial $P$ exhibiting the singularity, that the only quadrature for $\sigma$ with all nodes in $[a,b]$ is $\sigma$ itself. 

If $\sigma$ is not singular with respect to $[a,b]$ and $k$, then the same proof as above with minor modifications
shows that for any large enough $n$, there exist Chebyshev-type quadratures for $\sigma$ having degree at least $k$ and $n$ distinct nodes in $(a,b)$. If $k$ is odd, the only modification is that in the proof of Theorem~\ref{quadrature_exists_thm}, one should obtain its representations directly in $[a,b]$ without passing to the larger interval $[c,d]$ (this is possible since $\sigma$ is non-singular on $[a,b]$). Then the quadrature thus obtained will have all its nodes in $(a,b)$. For even $k$, the additional required modification is to first replace $\sigma$ by $\sigma'$, a canonical representation of it with support in $[a,b]$, index $k+2$ and the same first $k$ moments \cite[III \textsection 4]{KN77}, then to apply the above proof for $\sigma'$ and $k+1$.
\end{remark}

\end{subsection}
\begin{subsection}{Bounds for Chebyshev-type quadratures}\label{one_dim_quadrature_section}
In this section we prove Theorems~\ref{one_dim_quadrature_thm} and \ref{large_atom_upper_bound_theorem}, which give general upper bounds for the minimal number of nodes required in a Chebyshev-type quadrature, and Theorem~\ref{general_lower_bound}, Corollary~\ref{lower_bound_theorem} and Theorem~\ref{second_example_theorem} which give lower bounds for the required number of nodes and examples of cases where many nodes are required. We remark first about possible generalizations of our results.

\begin{remark}\label{main_theorem_extensions_remark}
\begin{enumerate}
\item It is sometimes desirable that the nodes of the quadrature be distinct and contained in the open interval $(0,1)$.
To obtain bounds for such formulas using our results, start by picking a small $\eps>0$, linearly map $\sigma$ to have support in $[\eps,1-\eps]$, apply Theorem~\ref{one_dim_quadrature_thm} to the new measure and use the freedom afforded by \eqref{p_condition} to make the moments of the resulting quadrature equal those of $\sigma$. To make the nodes distinct, use Theorem~\ref{Kuijlaars_distinct_nodes} (the proof of Theorem~\ref{one_dim_quadrature_thm} gives at least $k$ distinct nodes).


\item It may also be desirable to have a result similar to Theorem~\ref{one_dim_quadrature_thm} for functions other than $x^j$. The main ingredient required to adapt our proof to such a setting is to have a ``quantitative inverse mapping theorem'', as in Proposition~\ref{perturbation_of_moments_prop}, for the new collection of functions.
\end{enumerate}
\end{remark}
We start our proofs by recalling the definition \eqref{R_sigma_def} of $R_\sigma$ for a probability measure $\sigma$, and noting the following simple properties:
\begin{enumerate}
\item $R_\sigma$ is monotonically increasing.
\item $R_\sigma(\delta)=0$ if and only if $\sigma$ has an atom of mass at least $\delta$.
\item If $\sigma$ is supported on $[0,1]$, then $R_\sigma(\frac{1}{m})\le \frac{1}{m}$ for each $m\in\N$.
\item If $\sigma$ is absolutely continuous with a density that is essentially bounded by $M$, then $R_\sigma(\delta)\ge \frac{\delta}{M}$ for all $\delta$.
\end{enumerate}



\subsubsection{Proof of Theorem~\ref{one_dim_quadrature_thm}}
We start with a lemma providing a simple approximate Chebyshev-type quadrature.
\begin{lemma}\label{simple_approximation_lemma}
Let $\sigma$ be a probability measure with $\sigma([0,1])=1$. For $n\in\N$, let $\mu:=\frac{1}{n} \sum_{i=1}^n\delta_{y_i}$, where the $(y_i)$ are chosen according to the rule:
\begin{equation}
y_i := \min\left(y\in[0,1]\ \big|\ \sigma([0,y])\ge\frac{i}{n}\right)\qquad 1\le i\le n.\label{simple_approximation}
\end{equation}
Then for all $j\in\N$, we have
\begin{equation*}
\left|\int x^jd\sigma - \int x^jd\mu\right|\le \frac{1}{n}.
\end{equation*}
\end{lemma}
\begin{proof}
Define $y_0:=0$ and for each $0\le i\le n$ define the ``leftover mass at $y_i$'' by
\begin{equation*}
\alpha_i:=\sigma([0,y_i])-\frac{i}{n}.
\end{equation*}
Note that $0\le \alpha_i\le \sigma(\{y_i\})$ by definition of the $y_i$. Also, define measures $(\sigma_i)_{i=1}^n$ by
\begin{equation*}
\sigma_i(\cdot) := \sigma(\cdot\cap(y_{i-1},y_i])+\alpha_{i-1}\delta_{y_{i-1}}(\cdot)-\alpha_i\delta_{y_i}(\cdot)
\end{equation*}
and note that these measures are non-negative with total mass exactly $\frac{1}{n}$ and that we have $\sigma=\sum_{i=1}^n \sigma_i$. Now, fix $j\in\N$ and estimate
\begin{equation*}
\left|\int x^jd\sigma - \int x^j d\mu\right|=\left|\sum_{i=1}^n \int (x^j-y_i^j)d\sigma_i\right|\le \frac{1}{n}\sum_{i=1}^n \left(y_i^j-y_{i-1}^j\right)\le \frac{1}{n}
\end{equation*}
as required.
\end{proof}
\begin{remark}
It is worth noting that $\frac{c}{n}$ for some $c>0$ is the best approximation possible in this level of generality if one uses the above method of dividing $\sigma$ into $\sigma_i$'s with mass $\frac{1}{n}$ and approximating each one with one point. This can be seen by considering the example of $\sigma=\frac{1}{2}(\delta_0 + \delta_1)$ when $n$ is odd and the example of $\sigma=\frac{1}{3}(\delta_0+\delta_{\frac{1}{2}}+\delta_1)$ when $n$ is even.
\end{remark}

Our aim is to perturb the above simple approximation into a Chebyshev-type quadrature for $\sigma$. To this end, we define the moment map $T_k:\R^k\to\R^k$ by
\begin{equation*}
(T_k(z))_j:=\sum_{i=1}^k z_i^j,
\end{equation*}
and rely on the following quantitative ``inverse mapping theorem'':
\begin{proposition}\label{perturbation_of_moments_prop}
Fix $\rho>0$, integer $k\ge 2$ and let $z\in\R^k$ satisfy
\begin{equation}\label{z_separation_cond}
\frac{\rho}{3(k-1)}\le z_i\le 1-\frac{\rho}{3(k-1)} \quad\text{ and }\quad |z_i-z_{j}|\ge \frac{\rho}{k-1}
\end{equation}
for all $1\le i,j\le k$ with $j\neq i$. Then for any $p\in\R^k$ satisfying
\begin{equation}\label{close_moments_starting_point}
|p-T_k(z)|_\infty\le \frac{\rho}{3}\left(\frac{\rho}{12e}\right)^{k-1},
\end{equation}
there exists $w\in\R^k$ satisfying $|w-z|_\infty\le \frac{\rho}{3(k-1)}$ and $T_k(w)=p$.
\end{proposition}
In words, the proposition shows that if the $z_i$ are well separated, then the image through $T_k$ of a ball around $z$ contains a ball in moment space (where the balls are in the $l_\infty$ metric), and it gives quantitative bounds on the radii of these balls.
Since the proof of this proposition is somewhat long, we delay it until after we explain how Theorem~\ref{one_dim_quadrature_thm} follows from the proposition and lemmas. 

Iterating the proposition, we obtain the following corollary.
\begin{corollary}\label{disjoint_subsets_cor}
Given $\mu:=\frac{1}{n}\sum_{i=1}^n \delta_{y_i}$ with all $0\le y_i\le 1$ and $k\ge 2$. Suppose that there exist $0<\rho\le 1$ and $s$ disjoint subsets $(z(r))_{r=1}^s$ of the $(y_i)$'s, each of size exactly $k$, such that
\begin{equation}\label{separation_condition}
\begin{split}
&\frac{\rho}{3(k-1)}\le z(r)_i \le 1-\frac{\rho}{3(k-1)},\\
&|z(r)_i-z(r)_j|\ge \frac{\rho}{k-1}
\end{split}
\end{equation}
for all $1\le r\le s$ and $1\le i,j\le k$ with $j\neq i$. Then for any $p\in\R^k$ satisfying
\begin{equation}\label{perturbation_condition}
\left|p_j-\int x^jd\mu\right|\le \frac{\rho s}{3n}\left(\frac{\rho}{12e}\right)^{k-1} \qquad 1\le j\le k,
\end{equation}
there exists $\mu'$ of the form $\mu':=\frac{1}{n}\sum_{i=1}^{n} \delta_{x_i}$ with all $0\le x_i\le 1$, such that $\int x^jd\mu'=p_j$ for all $j$.
\end{corollary}
We emphasize that in this corollary and the rest of the proof, by disjoint subsets $z(r)\subseteq (y_i)$ we mean that we may choose indices $(i^r_1,\ldots, i^r_k)_{r=1}^s$ such that $z(r)_j=y_{i^r_j}$ and each $i$ appears at most once in all these index sets. Note that with this convention, if the $(y_i)$ contain a certain value multiple times, then it may happen that the $(z(r))$ also contain this value multiple times.
\begin{proof}
The corollary follows by applying Proposition~\ref{perturbation_of_moments_prop} to each of the subsets $z(r)$, each time changing the moments of the measure in the direction of the vector $p$. Note the additional factor $\frac{1}{n}$ in \eqref{perturbation_condition} as compared to \eqref{close_moments_starting_point}. This factor appears since $T_k$ is an unnormalized sum whereas $\mu'$ contains the normalization factor $\frac{1}{n}$.
\end{proof}
Finally, it remains to show that if a measure does not have large atoms, then the simple approximation of Lemma~\ref{simple_approximation_lemma} contains many disjoint subsets as in Corollary~\ref{disjoint_subsets_cor}.
\begin{lemma}\label{subsets_exist_lemma}
Let $\sigma$ be a probability measure with $\sigma([0,1])=1$. For $n\in\N$, let $(y_i)_{i=1}^n$ be the simple approximation \eqref{simple_approximation}. Then for each integer $k\ge 2$ such that $n\ge k(k+3)$ there exist $\lceil\frac{n}{k+3}\rceil$ disjoint subsets $(z(r))$ of the $(y_i)$, each of size exactly $k$, which satisfy \eqref{separation_condition} with $\rho=(k-1)R_\sigma(\frac{1}{k+3})$.
\end{lemma}
\begin{proof}
Define $y_0:=0$. Note that by definition, we have $y_j-y_i\ge R_\sigma(\frac{j-i}{n})$ for $0\le i<j\le n$. Let $i_0:=\lceil\frac{n}{k+3}\rceil$ and define the subsets $(z(r))_{r=1}^{i_0}$ by
\begin{equation*}
z(r)_j:=y_{ji_0+r-1}
\end{equation*}
for $1\le j\le k$. We need to verify that the conditions in \eqref{separation_condition} hold with $\rho=(k-1)R_\sigma(\frac{1}{k+3})$. To check the first condition, note that since $R_\sigma$ is non-decreasing and $\frac{n}{k+3}\le i_0\le\frac{n}{k+3}+1$, we have
\begin{equation*}
\begin{split}
z(r)_i&\ge z(1)_1=y_{i_0}\ge R_\sigma\left(\frac{i_0}{n}\right)\ge R_\sigma\left(\frac{1}{k+3}\right)\ge \frac{1}{3}R_\sigma\left(\frac{1}{k+3}\right),\\
z(r)_i&\le z(i_0)_k=y_{(k+1)i_0-1}\le 1-R_\sigma\left(1-\frac{(k+1)i_0-1}{n}\right)\le \\
&\le1-R_\sigma\left(1-\frac{k+1}{k+3}-\frac{k}{n}\right)\le 1-R_\sigma\left(\frac{1}{k+3}\right)\le 1-\frac{1}{3}R_\sigma\left(\frac{1}{k+3}\right)
\end{split}
\end{equation*}
using the assumption that $n\ge k(k+3)$. The second condition in \eqref{separation_condition} follows similarly.
\end{proof}
\begin{remark}
We note that there do exist $\sigma$ with atoms of size $\frac{1}{k+1}$ for which the $(y_i)$ of \eqref{simple_approximation} do not contain even one subset which satisfies the separation condition \eqref{separation_condition} for a positive $\rho$. For example, $\sigma=\frac{1}{k+1}\sum_{i=1}^{k+1}\delta_{\frac{i-1}{k}}$. Hence the above lemma is close to optimal.
\end{remark}
Putting all the above claims together, we may finish the proof of Theorem~\ref{one_dim_quadrature_thm}.
\begin{proof}[Conclusion of the proof of Theorem~\ref{one_dim_quadrature_thm}]
Let $\rho$ and $r$ be as in the theorem. Fix an integer $n\ge r^{-1}$ and a vector $p\in\R^k$ satisfying \eqref{p_condition}. By Lemma~\ref{simple_approximation_lemma}, we have $(y_i)_{i=1}^n\subseteq[0,1]$ such that for all $j\in \N$ we have
\begin{equation}\label{first_approx}
\left|\frac{1}{n}\sum_{i=1}^n y_i^j-\int x^jd\sigma(x)\right|\le \frac{1}{n}.
\end{equation}
Note that using the 3'rd property of $R_\sigma$ appearing in the beginning of the section, we have $\rho\le 1$ and so $n\ge k(k+3)$. Hence, by Lemma~\ref{subsets_exist_lemma}, there exist $s:=\lceil\frac{n}{k+3}\rceil$ disjoint subsets $(z(r))_{r=1}^s$ of the $(y_i)$, each of size exactly $k$, which satisfy \eqref{separation_condition} for the given $\rho$. Hence, by Corollary~\ref{disjoint_subsets_cor}, for any $p'\in\R^k$ satisfying
\begin{equation}\label{perturb_approx}
\left|p'_j - \frac{1}{n}\sum_{i=1}^n y_i^j\right|\le \frac{\rho s}{3n}\left(\frac{\rho}{12e}\right)^{k-1}
\end{equation}
for all $j$, there exist $(x_i)_{i=1}^n\subseteq[0,1]$ such that $\frac{1}{n}\sum_{i=1}^n x_i^j = p'_j$ for all $1\le j\le k$. Since $p$ satisfies \eqref{p_condition}, equations \eqref{first_approx} and \eqref{perturb_approx} will imply the theorem if
\begin{equation*}
r+\frac{1}{n}\le \frac{\rho s}{3n}\left(\frac{\rho}{12e}\right)^{k-1}.
\end{equation*}
This now follows by the definition of $s$ and the condition $n\ge r^{-1}$.
\end{proof}

\subsubsection*{Proof of Proposition \ref{perturbation_of_moments_prop}}
Fix an integer $k\ge 2$. We first define some notation: for $w\in\R^k$, let $V(w)$ be the Vandermonde matrix defined by
\begin{equation*}
V(w):=\begin{pmatrix}1&\cdots&1\\w_1&\cdots&w_k\\\vdots&\vdots&\vdots\\w_1^{k-1}&\cdots&w_k^{k-1}\end{pmatrix},
\end{equation*}
and let $U(w)$ be a slightly modified version defined by
\begin{equation*}
U(w):=\begin{pmatrix}1&\cdots&1\\2w_1&\cdots&2w_k\\\vdots&\vdots&\vdots\\kw_1^{k-1}&\cdots&kw_k^{k-1}\end{pmatrix}.
\end{equation*}
For a matrix $A\in M_{k\times k}$, define $\lVert A\rVert_{\infty}:=\max_{1\le i\le k} \sum_{j=1}^k |A_{ij}|$, the infinity norm of the matrix. We continue by citing (a special case of) a theorem of Gautschi about norms of inverses of Vandermonde matrices \cite{G62}.
\begin{theorem}\label{Vandermonde_estimate}(Gautschi)
For $w\in\R^k$ satisfying $w_i\ge 0$ for all $i$ and $w_i\neq w_j$ for all $i\neq j$, we have
\begin{equation*}
\lVert V(w)^{-1}\rVert_{\infty}=\max_{1\le i\le k}\prod_{\substack{j=1\\j\neq i}}^k\frac{1+w_j}{|w_i-w_j|}.
\end{equation*}
\end{theorem}
We immediately deduce
\begin{corollary}\label{perturbed_Vandermonde_norm_estimates}
For $w\in\R^k$ satisfying $w_i\ge 0$ for all $i$ and $w_i\neq w_j$ for all $i\neq j$, we have
\begin{equation*}
\lVert U(w)^{-1}\rVert_{\infty}\le \max_{1\le i\le k}\prod_{\substack{j=1\\j\neq i}}^k\frac{1+w_j}{|w_i-w_j|}.
\end{equation*}
In particular, if $0\le w_i\le 1$ for all $i$ and there exists $0<\sigma\le 1$ such that $|w_i-w_j|\ge \frac{\sigma}{k-1}$ for all $i\neq j$ then
\begin{equation*}
\lVert U(w)^{-1}\rVert_{\infty}\le \frac{1}{k}\left(\frac{4e}{\sigma}\right)^{k-1}.
\end{equation*}
\end{corollary}
\begin{proof}
Noting that $U(w) = DV(w)$ where $D$ is a diagonal matrix with $1,2,\ldots, k$ on its diagonal, we see that
\begin{equation*}
\lVert U(w)^{-1}\rVert_{\infty} \le \lVert V(w)^{-1}\rVert_{\infty}\lVert D^{-1}\rVert_{\infty}=\lVert V(w)^{-1}\rVert_{\infty},
\end{equation*}
so the first part of the corollary follows from Theorem \ref{Vandermonde_estimate}. For the second part, assume for simplicity that $k$ is odd. For the case $k=1$ there is nothing to prove, for $k\ge 3$ the assumptions and Stirling's approximation imply
\begin{equation*}
\begin{split}
\max_{1\le i\le k}\prod_{\substack{j=1\\j\neq i}}^k\frac{1+w_j}{|w_i-w_j|}&\le 2^{k-1}\left(\left(\frac{\sigma}{k-1}\right)^{k-1}\left[\left(\frac{k-1}{2}\right)!\right]^2\right)^{-1}\le\\
&\le \frac{1}{\pi(k-1)}2^{k-1}\left(\frac{2e}{\sigma}\right)^{k-1}\le\frac{1}{k}\left(\frac{4e}{\sigma}\right)^{k-1}.
\end{split}
\end{equation*}
Similarly, one can check that the required estimate holds when $k\ge 2$ is even.
\end{proof}
We continue the proof by defining a vector field $G:\R^k\to\R^k$ and an ODE,
\begin{equation*}
\begin{split}
G(w)&:=U(w)^{-1}(p-T_k(w)),\\
\dot{w}(t)&:=G(w(t)) \quad \text{ and } \quad w(0):=z.
\end{split}
\end{equation*}
By standard existence theorems for ODEs, 
there exists a solution to the ODE $w:[0,\tau_*)$ defined up to the first time that $G(w(t))$ is undefined, i.e., the first time that $w_i(t)=w_j(t)$ for some $i\neq j$. $\tau_*=\infty$ if such a time does not exist. Let also $t_*$ be the first time that $|w(t)-z|_\infty=\frac{\rho}{3(k-1)}$, or infinity if such a time does not exist. It is clear from the separation conditions \eqref{z_separation_cond} on the coordinates of $z$ that $t_*\le \tau_*$ with a strict inequality if $\tau_*<\infty$.

Note that the Jacobian $\frac{d}{dw} T_k(w) = U(w)$. Hence, for each $t<\tau_*$,
\begin{equation*}
\frac{d}{dt}(T_k(w(t))-p) = \frac{d}{dw} T_k(w(t)) \dot{w}(t)=p-T_k(w(t))
\end{equation*}
from which it follows that $T_k(w(t))-p=e^{-t}(T_k(z)-p)$. We deduce that if $t_*=\infty$, then since for $t<t_*$, $w(t)$ lies in a compact set, we may extract a subsequence of $w(t)$ converging to some $w$ with $|w-z|_\infty\le\frac{\rho}{3(k-1)}$. By continuity of $T_k$, this $w$ satisfies $T_k(w)=p$ as required. Hence, we assume, in order to get a contradiction, that $t_*<\infty$. We now calculate
\begin{multline*}
w(t_*)-z = \int_0^{t_*} \dot{w}(s)ds = \int_0^{t_*} U(w(s))^{-1}(p-T_k(w(s)))ds =\\
=\int_0^{t_*} e^{-s}U(w(s))^{-1}ds(p-T_k(z)).
\end{multline*}
Hence, noting that by Corollary \ref{perturbed_Vandermonde_norm_estimates} with $\sigma=\frac{\rho}{3}$ we have for $s\le t_*$ that
\begin{equation*}
\lVert U(w(s))^{-1}\rVert_{\infty}\le \frac{1}{k}\left(\frac{12e}{\rho}\right)^{k-1},
\end{equation*}
we obtain (using assumption \eqref{close_moments_starting_point})
\begin{multline*}
|w(t_*)-z|_{\infty}\le \frac{1}{k}\left(\frac{12e}{\rho}\right)^{k-1}\int_0^{t_*} e^{-s}ds |p-T_k(z)|_\infty<\\
<\frac{1}{k}\left(\frac{12e}{\rho}\right)^{k-1} \left(\frac{\rho}{12e}\right)^{k-1}\frac{\rho }{3}<\frac{\rho}{3(k-1)}
\end{multline*}
contradicting the definition of $t_*$. Thus the proposition is proven.

\subsubsection{Proof of Theorem~\ref{large_atom_upper_bound_theorem}}
Recalling the notation of Theorem~\ref{large_atom_upper_bound_theorem}, let us fix $0<\eps<1$ satisfying 
\begin{equation}\label{eps_condition_proof}
 \frac{\eps}{\sigma^t_\eps([0,1])}<\frac{2}{2k+7}
\end{equation}
and
\begin{equation}\label{n_definition}
 n\ge\max\left(\frac{1}{r\sigma^t_\eps([0,1])},\frac{2k+6}{\eps}\right).
\end{equation}
 As noted in Remark~\ref{atomic_theorem_remark}, condition \eqref{eps_condition_proof} implies that the right hand side of \eqref{n_definition} is finite. Let $A=\{x\ |\ \sigma(\{x\})>\frac{2k+7}{2k+6}\eps\}$ and define a measure
\begin{equation*}
 \sigma_{2,n}:=\sum_{x\in A} \frac{1}{n}\lfloor n(\sigma(\{x\})-\eps)\rfloor\delta_{x}.
\end{equation*}
In words, $\sigma_{2,n}$ has an atom for every atom $x\in A$ and the mass of this atom is the largest multiple of $\frac{1}{n}$ which is no larger than $\sigma(\{x\})-\eps$. Define also
\begin{equation*}
 \sigma_{1,n}:=\sigma - \sigma_{2,n}.
\end{equation*}
Then, by our definitions and \eqref{n_definition}, we have
\begin{equation}\label{measure_sandwich}
 \sigma_\eps^t(B)\le \sigma_{1,n}(B)\le \frac{2k+7}{2k+6}\sigma_{\eps}^t(B)
\end{equation}
for every Borel set $B$. 
We now let $q:=\sigma_{1,n}([0,1])$ so that $1-q=\sigma_{2,n}([0,1])$. Note that $q>0$ and is a multiple of $\frac{1}{n}$ by the definition of $\sigma_{2,n}$. Letting $\sigma_{1,n}':=\frac{\sigma_{1,n}}{q}$, we have
\begin{equation}\label{sigma_decomposition}
 \sigma = q\sigma_{1,n}'+\sigma_{2,n}.
\end{equation}
We claim that there exists a Chebyshev-type quadrature for $\sigma_{1,n}'$ of degree at least $k$ and having exactly $qn$ (not necessarily distinct) nodes in $[0,1]$ ($qn$ is an integer!). By Theorem~\ref{one_dim_quadrature_thm}, we know that such a quadrature exists if 
\begin{equation}\label{n_condition}
 qn\ge\frac{1}{r'}
\end{equation}
where
\begin{equation*}
\begin{split}
 \rho'&:=(k-1)R_{\sigma_{1,n}'}\left(\frac{1}{k+3}\right),\\
r'&:=\frac{\rho'}{6(k+3)}\left(\frac{\rho'}{12e}\right)^{k-1}.
\end{split}
\end{equation*}
By \eqref{measure_sandwich} we have that $\sigma_{1,n}'([x,y])\le \frac{2k+7}{2k+6}\sigma_{\eps}'([x,y])$ for any $x\le y$. Hence $R_{\sigma_{1,n}'}(\delta)\ge R_{\sigma_{\eps}'}(\frac{2k+6}{2k+7}\delta)$ for any $0<\delta<\frac{1}{2}$ and in particular $\rho'\ge \rho$ and consequently $r'\ge r$ ($\rho$ and $r$ were defined in the statement of the theorem). In addition, by \eqref{measure_sandwich}, we have that $q\ge \sigma_\eps^t([0,1])$. We conclude that \eqref{n_condition} holds by \eqref{n_definition}.

To finish, we have obtained a Chebyshev-type quadrature for $\sigma_{1,n}'$ of degree at least $k$,
\begin{equation*}
 \mu_1:=\frac{1}{qn}\sum_{i=1}^{qn} \delta_{x_i}
\end{equation*}
for some $\{x_i\}_{i=1}^{qn}\subseteq[0,1]$. Defining
\begin{equation*}
 \mu:=\sigma_{2,n} + \frac{1}{n}\sum_{i=1}^{qn} \delta_{x_i},
\end{equation*}
it is straightforward to check using \eqref{sigma_decomposition} that $\mu$ is a Chebyshev-type quadrature of degree at least $k$ for $\sigma$ having exactly $n$ (not necessarily distinct) nodes in $[0,1]$.

\subsubsection{Lower bounds for the number of nodes}\label{lower_bounds_sec}
In this section we prove Theorem~\ref{general_lower_bound}, Corollary~\ref{lower_bound_theorem} and Theorem~\ref{second_example_theorem}.
\begin{proof}[Proof of Theorem~\ref{general_lower_bound}] Fix an odd integer $k\ge 3$ and let $\sigma$ be a probability measure on $\R$ with $\sigma(\{0\})<1$ and  $\int |x|^k d\sigma(x)<\infty$. Denote $m_j:=\int x^j d\sigma(x)$ for $1\le j\le k$. If $m_k=0$, the theorem is trivial. If $m_k<0$, we define $\tilde{\sigma}$, the ``reflection through 0 of $\sigma$'', by $\tilde{\sigma}(A):=\sigma(-A)$ for measurable sets $A$. It is straightforward that $\int x^j d\tilde{\sigma}(x)=(-1)^jm_j$ for $1\le j\le k$ implying that the RHS of the bound \eqref{n_sigma_0_lower_bound} of the theorem is the same for $\sigma$ and $\tilde{\sigma}$. Since it is also straightforward that $n_{\tilde{\sigma}}^0(k)=n_\sigma^0(k)$, we see that it is sufficient to prove the theorem with $\tilde{\sigma}$ replacing $\sigma$. Noting that $\int x^kd\tilde{\sigma}(x)=-m_k>0$ (since $k$ is odd), we shall henceforth assume, WLOG, that $m_k>0$.

Set $a:=\frac{m_k}{m_{k-1}}>0$ (using that $m_{k-1}>0$ since $k$ is odd and $\sigma(\{0\})<1$). Suppose that, for some $n$, $\mu:=\frac{1}{n}\sum_{i=1}^n \delta_{x_i}$ for $(x_i)_{i=1}^n\subset \R$ is a Chebyshev-type quadrature formula for $\sigma$ of degree at least $k$. Letting $f(x):=x^{k-1}(a-x)$, we then have that
\begin{equation}\label{f_int_zero}
\int f(x)d\mu(x) = \int f(x) d\sigma(x) = am_{k-1}-m_k = 0.
\end{equation}
We continue by noting that, since $k$ is odd, $f(0)=0$ and $f(x)>0$ for all $x\in (-\infty,a)\setminus\{0\}$. Thus, \eqref{f_int_zero} implies that either $\mu=\delta_0$ or $\mu([a,\infty))>0$. However, the former option is impossible since $\sigma\neq\delta_0$, which implies $m_2>0$, and $\mu$ has the same second moment as $\sigma$. It follows that if we denote $\xi:=\max\{x_i\}_{i=1}^n$, then $\xi\ge a>0$. Denoting now $g(x):=\frac{x^{k-1}}{a^{k-1}}$, we have that
\begin{equation*}
\int g(x)d\mu(x) = \int g(x) d\sigma(x) = \frac{m_{k-1}}{a^{k-1}}=\frac{m_{k-1}^k}{m_k^{k-1}}.
\end{equation*}
However, since $g(x)\ge 0$ for all $x$ (using that $k$ is odd) and $g(\xi)\ge g(a)=1$ (using that $\xi\ge a>0$), it follows that
\begin{equation*}
\frac{1}{n}=\frac{1}{n}g(a) \le \frac{1}{n}g(\xi)\le \int g(x)d\mu(x) = \frac{m_{k-1}^k}{m_k^{k-1}},
\end{equation*}
whence $n\ge \frac{m_k^{k-1}}{m_{k-1}^k}$ as required.
\end{proof}

\begin{proof}[Proof of Corollary~\ref{lower_bound_theorem}] Let $d\sigma(x):=1_{[0,\infty)}(x)\exp(-x)dx$ be the exponential distribution. Recall that $\sigma([x,\infty))=\exp(-x)$ for $x\ge 0$ and that $\int x^jd\sigma(x)=j!$ for $j\in\N$. Fix an odd integer $k\ge 3$ and define a new measure $\sigma_k'$ by
\begin{equation*}
\sigma_k'(A):=c_k\sigma(A\cap [0,2k])
\end{equation*}
for measurable sets $A$, where $c_k:=(1-\exp(-2k))^{-1}$ is chosen so that $\sigma_k'$ is a probability distribution. Define also the rescaling, $\sigma_k$, of $\sigma_k'$ to the interval $[0,1]$ by
\begin{equation*}
\sigma_k(A):=\sigma_k'(A*2k)
\end{equation*}
for measurable sets $A$, where $A*2k:=\{2kx\ |\ x\in A\}$. Noting that $\sigma_k$ is absolutely continuous, supported on $[0,1]$ and having density bounded above by $Ck$ for some $C>0$, we claim that if $k$ is sufficiently large, $\sigma_k$ satisfies the corollary (for that $k$). To see this, we shall prove below that if $k$ is sufficiently large then
\begin{equation}\label{moment_bounds_for_sigma_k}
\frac{c_k}{2} j! \le \int x^jd\sigma_k'(x) \le c_k j! \qquad \text{for $j=k-1$ and $j=k$}.
\end{equation}
From these inequalities we deduce, using that $c_k\to 1$ as $k\to\infty$ and that by Stirling's formula, $(k-1)!\sim \sqrt{2\pi k}\left(\frac{k-1}{e}\right)^{k-1}$ as $k\to\infty$,
\begin{equation*}
\frac{(\int x^k d\sigma_k)^{k-1}}{(\int x^{k-1}d\sigma_k)^k}=\frac{(\int x^k d\sigma'_k)^{k-1}}{(\int x^{k-1}d\sigma'_k)^k}\ge \frac{k^{k-1}}{c_k 2^{k-1}(k-1)!}\sim\sqrt{\frac{2}{\pi k}}\left(\frac{e}{2}\right)^k>\frac{1}{2\sqrt{k}}\left(\frac{e}{2}\right)^k
\end{equation*}
as $k\to\infty$. Thus, by Theorem~\ref{general_lower_bound}, for sufficiently large $k$,
\begin{equation*}
n_{\sigma_k}^0(k)\ge \frac{1}{2\sqrt{k}}\left(\frac{e}{2}\right)^k
\end{equation*}
as required. It remains only to prove \eqref{moment_bounds_for_sigma_k}. The second inequality of \eqref{moment_bounds_for_sigma_k} follows from the fact that $\int x^jd\sigma(x)=j!$. To see the first inequality, note first that
\begin{equation}\label{j_moment_equality}
\int x^j d\sigma_k(x) = c_k\left(j! - \int_{2k}^\infty x^j d\sigma(x)\right).
\end{equation}
Second, note that for $x\ge 2k$ and $j\le k$ we have $x^j\le (2k)^j\exp\left(\frac{x-2k}{2}\right)$, which can be seen by taking logarithms and differentiating. Thus, for $j=k-1$ and $j=k$, if $k$ is sufficiently large,
\begin{equation*}
\int_{2k}^\infty x^j d\sigma(x)\le \frac{(2k)^j}{e^k}\int_{2k}^\infty \exp\left(-\frac{x}{2}\right)dx\le \frac{2(2k)^j}{e^{2k}}\le \frac{j!}{2}
\end{equation*}
which, when plugged into \eqref{j_moment_equality}, proves the first inequality.
\end{proof}

\begin{proof}[Proof of Theorem~\ref{second_example_theorem}]
The bounds \eqref{example_n_sigma_0_bound} and \eqref{example_n_sigma_bound} follow directly from the bounds \eqref{even_n_bound} and \eqref{odd_n_bound} and the definitions of $d_n(\sigma_0), n_{\sigma_0}^0(k)$ and $n_{\sigma_0}(k)$. The bound \eqref{even_n_bound} follows from Bernstein's Theorem~\ref{Bernstein_CTQ} by replicating the Chebyshev-type quadrature given by the upper bound of that theorem to each of the two intervals in the support of $\sigma_0$. To see \eqref{odd_n_bound}, we will need the following definition and theorems. We call a Chebyshev-type quadrature formula \eqref{Chebyshev_type_quadrature} symmetric if the measures $\frac{1}{n}\sum_{i=1}^n\delta_{x_i}$ and $\frac{1}{n}\sum_{i=1}^n\delta_{-x_i}$ are equal. We will use a special case of a Theorem of F\"orster and Ostermeyer \cite{FO86}.
\begin{theorem}(\cite[Section 3, Corollaries 1 and 2]{FO86}\label{Forster_Ostermeyer_theorem}
If a probability measure $\sigma$ on a bounded interval has a density $w(x)$ satisfying $w(x)=w(-x)$ for all $x$, then for each $n\in\N$, there exists a \emph{symmetric} Chebyshev-type quadrature having exactly $n$ (not necessarily distinct) nodes and degree of accuracy $d_n(\sigma)$.
\end{theorem}
We will also use a special case of a Theorem of Peherstorfer \cite{P90}.
\begin{theorem}(\cite[special case of Theorem 3.1]{P90}\label{Peherstorfer_theorem}
There exists $C>0$ such that for each $n\in \N$, if a Chebyshev-type quadrature \eqref{Chebyshev_type_quadrature} for $\sigma_0$ has $n$ (not necessarily distinct) nodes $(x_i)_{i=1}^n$ satisfying $0\in(x_i)_{i=1}^n$, then its degree of accuracy $k$ satisfies $k\le C\ln(Cn)$.
\end{theorem}
We remark that the proof of Theorem~\ref{Peherstorfer_theorem} proceeds by taking a polynomial $T(x)$ of degree at most $k$ which satisfies $|T(x)|\le 1$ on the support of $\sigma_0$ and which is positive and grows very fast off the support of $\sigma_0$ (a variant of the Chebyshev polynomial may be used). Using the facts that the integral of $T$ with respect to $\sigma_0$ is at most $1$, that this integral must equal the integral of $T$ with respect to the given quadrature formula and that each node of the quadrature formula has weight $\frac{1}{n}$, one deduces that $n$ must be very large, to offset the contribution of the quadrature nodes outside the support of $\sigma_0$.

Now fix an odd integer $n\in \N$. By Theorem~\ref{Forster_Ostermeyer_theorem}, there exists a \emph{symmetric} Chebyshev-type quadrature \eqref{Chebyshev_type_quadrature} for $\sigma_0$ having exactly $n$ nodes and algebraic degree of accuracy $d_n$. Since $n$ is odd, the symmetry implies that $0$ is one of the nodes of this formula. This implies, by Theorem~\ref{Peherstorfer_theorem}, that $d_n(\sigma_0)\le C\ln(Cn)$ for some $C>0$, proving \eqref{odd_n_bound}.
\end{proof}

\begin{remark}
As a final remark for this section, we note that it is possible to have a sequence of absolutely continuous distributions $\sigma_k$ with $n_{\sigma_k}^0(k)$ rising as quickly as we want with $k$. However, the densities of these distributions will have very large essential supremums. For example, for $k=2m-1$, we can take a distribution with $m$ atoms and with the leftmost atom as small as we want. By Bernstein's theorem~\ref{Bernstein_estimate}, any Chebyshev-type quadrature for it of degree at least $k$ will have at least as many nodes as one over that atom (since the distribution and its Gaussian quadrature coincide in this case). Now, we can convolve this distribution with a smooth function which is very close to a delta measure to obtain an absolutely continuous distribution whose Gaussian quadrature is as close as we want to the atomic measure (in the weak topology), so that Bernstein's theorem implies the result.
\end{remark}
\end{subsection}

\begin{subsection}{Random Chebyshev-type cubatures on the cube}\label{random_cubatures_on_cube_section}


In this section, we prove Lemma~\ref{small_ball_to_cubature_lemma} and Theorem~\ref{positive_density_thm}.
\begin{proof}[Proof of Lemma~\ref{small_ball_to_cubature_lemma}]
Suppose that for some $n,k\ge 1$ and all $\eps>0$ we have $p_{n,k,\eps}(\sigma)>0$. If $k=1$, a Chebyshev-type cubature always exists for $\sigma$ (placing all nodes on the mean of $\sigma$). Assume $k\ge 2$ and fix a sequence $\eps_j\to0$. Since $p_{n,k,\eps_j}>0$ we can find a measure
\begin{equation*}
 \sigma_j:=\frac{1}{n}\sum_{i=1}^n\delta_{x_i^{(j)}}
\end{equation*}
such that $\Vert M_k(\sigma) - M_k(\sigma_j)\Vert_\infty\le \eps_j$. These measures must have a converging subsequence as $j\to\infty$ (in the sense that the location of the atoms converges) to some $\sigma':=\frac{1}{n}\sum_{i=1}^n\delta_{x_i}$ since if any of the atoms goes to infinity we necessarily have $\Vert M_k(\sigma_j)\Vert_\infty\to\infty$ since $k\ge 2$ and each atom carries a fixed weight $\frac{1}{n}$. $\sigma'$ is the required cubature.
\end{proof}
We proceed to prove Theorem~\ref{positive_density_thm}. Recalling the statement of the theorem, we first observe that since the $M_i$ are IID vectors in $\R^{\polydim(k,d)}$,
the central limit theorem gives that $\bar{S}_n$ converges weakly to a $N(0,\Sigma)$ RV for some matrix $\Sigma$. To prove the proposition we would like to show that $\Sigma$ is positive definite and that a local limit theorem also holds. This will imply that for large enough $n$, the density of $\bar{S}_n$ exists and is uniformly close to that of $N(0,\Sigma)$, whence it is uniformly positive in a neighborhood of the origin.

For a random variable $X\in\R^m$ we write $\hat{X}:\R^m\to\C$ for the characteristic function $f(\lambda):=\E e^{i\lambda\cdot X}$. We use the following local limit theorem from \cite{BR76}.
\begin{theorem}(\cite[Th. 19.1, Ch. 4]{BR76})
Let $(X_n)_{n\ge 1}$ be a sequence of IID random vectors in $\R^m$ with $\E X_1=0$ and positive definite $\Sigma:=\Cov(X_1)$. Let $Q_n:=\frac{1}{\sqrt{n}}(X_1+\cdots X_n)$, then the following are equivalent:
\begin{enumerate}
\item $\hat{Q}_1\in L^p(\R^m)$ for some $1\le p<\infty$.
\item For every sufficiently large $n$, $Q_n$ has a density $q_n$ and
\begin{equation*}
\lim_{n\to\infty} \sup_{x\in\R^m} |q_n(x)-\phi_{0,\Sigma}(x)|=0
\end{equation*}
where $\phi_{0,\Sigma}$ is the density of a $N(0,\Sigma)$ random vector.
\end{enumerate}
\end{theorem}

In our case we take $X_i:=M_i-\E M_1$ and we will show that
\begin{equation} \label{L_p_estimate}
\hat{M}_1\in L^p(\R^{\polydim(k,d)}) \qquad \text{for some $1\le p<\infty$.}
\end{equation}
Note that to use the above theorem it may seem necessary to separately show that $\Sigma:=\Cov(M_1)$ is positive definite, but this also follows from \eqref{L_p_estimate} since if $\Cov(M_1)$ were singular then $X_1$ would be supported in a linear subspace and \eqref{L_p_estimate} would not hold, since in that case $\hat{X}_1(\mu+\lambda)$ would equal $\hat{X}_1(\mu)$ for every $\lambda$ orthogonal to that linear subspace.

Hence Theorem \ref{positive_density_thm} will follow by verifying \eqref{L_p_estimate}. Such estimates are standard in the theory of oscillatory integrals but since we could not find this exact result, we prove it using standard methods from the book \cite{S93} by Stein. Following that book, we use the next estimate of Van der Corput to prove what we need.
\begin{proposition}\label{VdC_estimate}(\cite[Prop. 2, Ch. VIII]{S93})
Suppose $\phi:(a,b)\to\R$ is smooth and satisfies $|\phi^{(j)}(\rho)|\ge 1$ for $\rho\in(a,b)$. Then
\begin{equation*}
\left|\int_a^b e^{i\lambda\phi(\rho)}d\rho\right|\le C_j\lambda^{-1/j}
\end{equation*}
when $j\ge 2$ or when $j=1$ and $\phi'$ is monotonic. The bound $C_j$ is independent of $\phi$, $\lambda$, $a$ and $b$.
\end{proposition}
For the case $j=1$ we will not be able to ensure monotonicity, so we will use instead:
\begin{lemma}\label{oscil_estimate_k_one}
Suppose $\phi:(a,b)\to\R$ is smooth with $|\phi'(\rho)|\ge 1$ for $\rho\in(a,b)$. then
\begin{equation}
\left|\int_a^b e^{i\lambda\phi(\rho)}d\rho\right|\le \frac{2}{\lambda}+\frac{b-a}{\lambda}\max_{\rho\in(a,b)}|\phi''(\rho)|.
\end{equation}
\end{lemma}
\begin{proof}
The proof is a slight variation on the proof of the previous proposition for the case $k=1$, as it appears in \cite{S93}. Using integration by parts,
\begin{equation*}
\int_a^b e^{i\lambda\phi(\rho)}d\rho = \int_a^b e^{i\lambda\phi(\rho)}\frac{i\lambda\phi'(\rho)}{i\lambda\phi'(\rho)}d\rho = \frac{e^{i\lambda\phi(\rho)}}{i\lambda\phi'(\rho)}|_a^b-\int_a^b e^{i\lambda\phi(\rho)}\frac{d}{d\rho}\left(\frac{1}{i\lambda\phi'(\rho)}\right)d\rho.
\end{equation*}
The boundary terms are majorized by $\frac{2}{\lambda}$ and the second term satisfies
\begin{equation*}
\left|\int_a^b e^{i\lambda\phi(\rho)}\frac{d}{d\rho}\left(\frac{1}{i\lambda\phi'(\rho)}\right)d\rho\right| \le \frac{1}{\lambda}\int_a^b \frac{|\phi''(\rho)|}{|\phi'(\rho)|^2}d\rho\le \frac{b-a}{\lambda}\max_{\rho\in(a,b)}|\phi''(\rho)|.\qedhere
\end{equation*}
\end{proof}

For $u\in \S^{d-1}$, let $D_u$ denote the directional derivative operator in the direction $u$, and let $D_u^j$ be its $j$-th power; i.e.,
\begin{equation*}
D_u^j(f)(x)=\frac{d^j}{d\rho^j}f(x+\rho u)|_{\rho=0}.
\end{equation*}
We continue with two simple technical lemmas:
\begin{lemma}
Let $Q:\R^d\to\R$ be a non-zero polynomial of degree $j$, then there exists $u\in\S^{d-1}$ such that $D_u^j(Q)$ is a non-zero constant function.
\end{lemma}
\begin{proof}
Denote $m:=\binom{j+d-1}{d-1}$ and let $\tilde{P}_j^d:\R^d\to\R^m$ be defined by $\tilde{P}_j^d(x):=(x^\alpha)_{|\alpha|=j}$,
where $\alpha$ is a multi-index. We first note that the image of $\S^{d-1}$ under $\tilde{P}_j^d$ is not contained in any proper linear subspace of $\R^m$. This follows since otherwise there would exist $\eta\in \S^{m-1}$ such that $\eta\cdot \tilde{P}_j^d(u)=0$ for all $u\in\S^{d-1}$ contradicting the fact that $\eta\cdot\tilde{P}_j^d$ is a non-zero homogeneous polynomial.


Now decompose $Q$ as $Q=Q_1+Q_2$ where $Q_1$ is a non-zero homogeneous polynomial of degree $j$ and $Q_2$ is of degree at most $j-1$. Write $Q_1(x)=\sum_{|\alpha|=j} a_\alpha x^{\alpha}$. It follows from the above that we may choose $u\in \S^{d-1}$ such that $\tilde{P}_j^d(u)$ is not orthogonal to $(a_\alpha)_{|\alpha|=j}$. Hence taking $\rho\in\R$, we see that $Q(\rho u)$ is a non-zero polynomial of degree $j$ in $\rho$, from whence it follows that for every $x\in\R^d$, $Q(x+\rho u)$ is a polynomial of degree $j$ in $\rho$ with the same leading coefficient. Finally, we deduce that $\frac{d^j}{d\rho^j}Q(x+\rho u)|_{\rho=0}$ is a non-zero constant function as required.
\end{proof}
\begin{lemma} \label{polynomial_derivative_bound}
There is $c_{k}>0$ such that for every direction $\eta\in \S^{\polydim(k,d)-1}$ the function $\eta\cdot P_k^d$ satisfies that there exists $1\le j\le k$ and $u\in \S^{d-1}$ with
\begin{equation*}
\min_{x\in[-1,1]^d} |D^j_{u}(\eta\cdot P_k^d)(x)|\ge c_k.
\end{equation*}
\end{lemma}
\begin{proof}
Fix $\eta\in \S^{\polydim(k,d)-1}$ and denote $P(x):=\eta\cdot P^d_k(x)$. Since $P(x)$ is a non-zero polynomial of some degree $j\le k$, by the previous lemma, there exists a $u\in \S^{d-1}$ such that $D^j_{u} P$ is a non-zero constant. Hence, in particular, $\min_{x\in[-1,1]^d} |D^j_{u}(P)(x)|>0$. The lemma follows since
\begin{equation*}
\max_{u\in \S^{d-1}}\max_{1\le j\le k}\min_{x\in[-1,1]^d} |D^j_{u}(\eta\cdot P_k^d)(x)|
\end{equation*}
is a continuous function of $\eta$ and $\S^{\polydim(k,d)-1}$ is a compact set.
\end{proof}

\begin{proof}[Proof of Theorem \ref{positive_density_thm}]
Denote $f(\lambda):=\hat{M}_1 (\lambda)$. Fix a direction $\eta\in \S^{\polydim(k,d)-1}$, let $r>0$ and consider
\begin{equation} \label{f_eta_dir_eq}
f(r\eta) = \E e^{i r\eta\cdot M_1} = \int_{[-1,1]^d} e^{ir\eta\cdot P_k^d(x)}dx.
\end{equation}
Our goal is to prove an estimate of the form
\begin{equation}\label{directional_Fourier_estimate}
|f(r\eta)|\le \hat{C}_kr^{-s}
\end{equation}
for some $\hat{C}_k$ and $s>0$ independent of $\eta$ and $r$. Such an estimate will imply \eqref{L_p_estimate} for $p>\frac{\polydim(k,d)}{s}$.

Applying Lemma \ref{polynomial_derivative_bound}, we obtain a number $1\le j\le k$ and a direction $u\in \S^{d-1}$ such that
\begin{equation}\label{poly_derivative_lower_bound}
\min_{x\in[-1,1]^d} |D^j_{u} (\eta\cdot P^d_k)(x)|\ge c_{k},
\end{equation}
where $c_k>0$ is independent of $\eta$. In addition, we may define
\begin{equation}\label{poly_derivative_upper_bound}
\bar{C}_k:=\max_{\eta\in \S^{\polydim(k,d)-1}}\max_{x\in[-1,1]^d} \max_{u\in \S^{d-1}} |D^2_u(\eta\cdot P^d_k)(x)|<\infty.
\end{equation}
We now decompose our space to the line $H=\{\rho u\}_{\rho\in\R}$ and $H^\perp$. We say that $y\in H^{\perp}$ is \emph{contributing} if there exists $\rho\in\R$ such that $\rho u+y\in[-1,1]^d$. For contributing $y$'s define
\begin{equation*}
\begin{split}
a_y &= \min\{\rho\in\R\ |\ \rho u+y\in[-1,1]^d\},\\
b_y &= \max\{\rho\in\R\ |\ \rho u+y\in[-1,1]^d\}.
\end{split}
\end{equation*}
For non-contributing $y$'s set $a_y=b_y=0$. Note that by a simple $l_2$ estimate, if $|y|>\sqrt{d}$, then $y$ is non-contributing. We note that we may estimate the integral \eqref{f_eta_dir_eq} we are after as
\begin{equation*}
|f(r\eta)| = \left|\int_{H^{\perp}}\int_{a_y}^{b_y} e^{ir\eta\cdot P_k^d(\rho u+y)}d\rho dy\right|\le \int_{H^\perp\cap[-\sqrt{d},\sqrt{d}]^d}\left|\int_{a_y}^{b_y} e^{ir\eta\cdot P_k^d(\rho u+y)}d\rho\right| dy.
\end{equation*}
Using the inequalities \eqref{poly_derivative_lower_bound} and \eqref{poly_derivative_upper_bound} we may apply Proposition \ref{VdC_estimate} and Lemma \ref{oscil_estimate_k_one} to $\phi_y(\rho):=(\eta \cdot P_k^d)(\rho u +y)$ to obtain
\begin{equation}
\left|\int_{a_y}^{b_y} e^{ir\phi_y(\rho)}d\rho\right|\le \tilde{C}_{k}r^{-1/j} \le \tilde{C}_{k} r^{-1/k},
\end{equation}
for $\tilde{C}_k$ independent of $\eta$, $r$ and $y$. Plugging this estimate in the previous integral we finally obtain
\begin{equation*}
|f(r\eta)|\le \hat{C}_{k} r^{-1/k}
\end{equation*}
as required.
\end{proof}
\end{subsection}

\begin{subsection}{Local cubature formulas}
In this section we prove Theorems~\ref{special_sphere_cubature} and \ref{special_cubature_on_cyl_thm}. The idea behind the proof of Theorem~\ref{special_sphere_cubature} is to present the sphere in spherical coordinates. Partition the spherical coordinate space into suitable boxes and then use the fact that the measure on each box is a product measure to construct Chebyshev-type quadratures for them using our one-dimensional quadrature results. The same idea with a few variations was used in Wagner \cite{W92}. To prove Theorem~\ref{special_cubature_on_cyl_thm}, we use the fact that the measure on the cylinder is a product of the measures on the $x$-axis and the measure on the sphere. We then partition the $x$-axis to small intervals and construct a product Chebyshev-type cubature on each interval using our construction for the sphere and our one-dimensional quadrature results.

\begin{subsubsection}{Sphere Construction}
We begin the proof by introducing spherical coordinates. Let $\ang_{d}:=(\phi,\theta_1,\ldots, \theta_{d-1})\in\R^{d}$ and $\Omega_{d}:=\{0\le \phi\le 2\pi, 0\le \theta_i\le \pi\text{ for $1\le i\le d-1$}\}$. Then define $T:\Omega_{d}\to\R^{d+1}$ (formally $T_{d}$) by
\begin{equation*}
\begin{split}
T(\ang_{d})_1 &:= \sin(\phi)\prod_{i=1}^{d-1}\sin(\theta_i),\\
T(\ang_{d})_2 &:= \cos(\phi)\prod_{i=1}^{d-1}\sin(\theta_i),\\
T(\ang_{d})_j &:= \cos(\theta_{j-2})\prod_{i=j-1}^{d-1}\sin(\theta_i) \quad \text{for $3\le j\le d+1$}.
\end{split}
\end{equation*}
This is a continuous and onto mapping of $\Omega_d$ to $\S^{d}$. Further endowing $\Omega_{d}$ with the measure
\begin{equation*}
d\mu_{d}(\ang_{d}) := d\phi\prod_{i=1}^{d-1} \sin^i(\theta_i)d\theta_i,
\end{equation*}
the map becomes measure preserving ($\S^d$ is endowed with the surface area measure $\sigma_d$). We will embed $\Omega_{d}$ into $\Omega_{d+1}$ and write (with slight abuse of notation) $\ang_{d+1} = (\ang_{d}, \theta_{d})$. Note also that $d\mu_{d+1}(\ang_{d+1}) = \sin^{d}(\theta_{d}) d\theta_{d}d\mu_{d}(\ang_{d})$. Similarly, we will embed $\S^d = T(\Omega_d)$ into $\S^{d+1}=T(\Omega^{d+1})$ by $T(\ang_{d+1}) = (\sin(\theta_{d})T(\ang_{d}), \cos(\theta_{d}))$.

We now construct the partition we shall use in Theorem~\ref{special_sphere_cubature}. In the spherical coordinates space $\Omega_d$, the sets of the partition will be taken as boxes, that is, Cartesian products of intervals.
\begin{proposition}\label{sphere_partition_to_rectangles}
For each $d\ge 1$ and $0<\tau<1$ there exist $C(d),c(d)>0$ (independent of $\tau$), $K=K(\tau,d)>0$, and a partition of $\Omega_{d}$ (up to measure $0$) into boxes $D_1,\ldots, D_K$ with side lengths smaller than $1$, $\diam(T(D_i))\le C(d)\tau$ and $\mu_{d}(D_i)\ge c(d)\tau^{d}$ for all $i$.
\end{proposition}
\begin{proof}
We proceed by induction. For $d=1$, we partition $\Omega_1=[0,2\pi]$ into $\lceil\frac{2\pi}{\tau}\rceil$ intervals of length $2\pi/\lceil\frac{2\pi}{\tau}\rceil<1$. It is straightforward to see that the required properties hold. Assume that the proposition holds for dimension $d-1$. We will construct boxes $D_1, \ldots, D_K$ satisfying the required properties for dimension $d$. First partition $[0,\pi]$ into $m:=\lceil\frac{\pi}{\tau}\rceil$ length $\frac{\pi}{m}<1$ intervals $(I_i)_{i=0}^{m-1}$ (overlapping in their end points) by $I_i:=[a_i,a_{i+1}]$ with $a_i:=i\frac{\pi}{m}$. For each $0\le i\le m-1$ we will define a set $\CC_i$ of boxes of the form $\tilde{D}\times I_i$ where $\tilde{D}\subseteq \Omega_{d-1}$ is a box. Then $D_1,\ldots, D_K$ will be the union of all of the $\CC_i$.

Fix $0\le i\le m-1$ and define 
$r:=\sin(\frac{a_i+a_{i+1}}{2})$ and $\tau':=\min(\frac{\tau}{r},\frac{1}{2})$. We then have for any $\alpha>0$,
\begin{equation}\label{r_tau_prime_ineq}
\begin{split}
&\max_{a_i\le \theta\le a_{i+1}}\sin(\theta)\le Cr,\\
&\int_{a_i}^{a_{i+1}} \sin^{\alpha}(\theta)d\theta\ge \tilde{c}(\alpha)r^{\alpha}\tau,\\
&c\tau\le r\tau'\le \tau
\end{split}
\end{equation}
for some $C,c>0$ independent of all other parameters and $\tilde{c}(\alpha)>0$ depending only on $\alpha$. Indeed, the first inequality follows from the facts that $0\le a_{i}<a_{i+1}\le \pi$ and $\sin$ is non-negative, concave and continuously differentiable on this interval. The second follows from these facts and $|a_{i+1}-a_i|=\frac{\pi}{m}\ge \frac{\pi}{\frac{\pi}{\tau}+1}\ge \frac{3}{4}\tau$. Finally, the RHS of the third inequality follows directly from the definition of $\tau'$ and the LHS relies on the facts that $\sin x\ge \frac{5}{6}x$ for $0\le x\le 1$ and $|a_{i+1}-a_i|\ge\frac{3}{4}\tau$ to deduce that $r\ge\frac{15}{48}\tau$, from which the inequality follows.

Using the induction hypothesis, let $\tilde{D}_1, \ldots, \tilde{D}_{\tilde{K}}$ be the partition of $\Omega^{d-1}$ which satisfies the proposition for $\tau'$. The set $\CC_i$ is the set $(\tilde{D}_j\times I_i)_{j=1}^{\tilde{K}}$. Fix $1\le j\le \tilde{K}$ and let $D:=\tilde{D}_j\times I_i$. It remains to check that $\diam(T(D))\le C(d)\tau$ and $\mu_d(D)\ge c(d)\tau^d$ for some $C(d),c(d)>0$ independent of $\tau$. To check the former, note that $T(D)=\{(\sin(\theta_{d-1})T(\tilde{D}_j), \cos(\theta_{d-1}))\ |\ a_i\le \theta_{d-1}\le a_{i+1}\}$ and hence by the triangle inequality, the induction hypothesis and \eqref{r_tau_prime_ineq}, we have
\begin{multline*}
\diam(T(D))\le \max_{a_i\le \theta_{d-1}\le a_{i+1}}\sin(\theta_{d-1})\diam(T(\tilde{D}_j))+\cos(a_i)-\cos(a_{i+1})\le\\
\le Cr\cdot C(d-1)\tau'+C\tau\le C(d)\tau.
\end{multline*}
To check the second bound, note that by the product structure of the measure and \eqref{r_tau_prime_ineq}, we have
\begin{equation*}
\mu_d(D) = \mu_{d-1}(\tilde{D}_j)\int_{a_i}^{a_{i+1}} \sin^{d-1}(\theta_{d-1})d\theta_{d-1}\ge c(d-1)(\tau')^{d-1}\cdot\tilde{c}(d-1)r^{d-1}\tau\ge c(d)\tau^d
\end{equation*}
as required.
\end{proof}
For the subsets $\{E_i\}$ of Theorem~\ref{special_sphere_cubature} we take $E_i:=T(D_i)$ where $\{D_i\}_{i=1}^K$ are the boxes of Proposition~\ref{sphere_partition_to_rectangles} (with the same $d$ and $\tau$ as in the theorem). For the rest of the proof fix $1\le i\le K$ and, for brevity, denote $D:=D_i$ and $E:=E_i$. Let $h:\R^{d+1}\to\R$ be defined by $h(z):=z^\alpha$ for a multi-index $\alpha$ with $|\alpha|\le k$. Since $D$ is a box, we may write $D:=J\times I_1\times\cdots\times I_{d-1}$ (these $\{I_i\}$ are different from those used in the proof of Proposition~\ref{sphere_partition_to_rectangles}). Note that
\begin{multline}\label{polynomial_transformation_to_sphere}
\int_E h(z)d\sigma_d(z) = \int_D h(T(\ang_d))d\mu_{d}(\ang_{d}) =\\
=\int_D \sin^{\alpha_1}(\phi) \cos^{\alpha_2}(\phi) d\phi\prod_{q=1}^{d-1}\sin^{\sum_{j=1}^{q+1}\alpha_j}(\theta_q) \cos^{\alpha_{q+2}}(\theta_q) \sin^q(\theta_q) d\theta_q=\\
=\int_J \sin^{\alpha_1}(\phi) \cos^{\alpha_2}(\phi) d\phi\prod_{q=1}^{d-1}\int_{I_q}\sin^{q+\sum_{j=1}^{q+1}\alpha_j}(\theta_q) \cos^{\alpha_{q+2}}(\theta_q) d\theta_q.
\end{multline}
We begin the construction of our cubature formula by constructing quadratures for the intervals $J$ and $I_q$.

\begin{lemma}\label{one_dim_quadrature_for_sine_lemma}
Given $0\le \gamma\le 1$, integers $k\ge 1$, $0\le q\le d-1$ and any $0\le \sigma_1<\sigma_2\le \tau_q$, where $\tau_0:=2\pi$ and $\tau_q:=\pi$ for $1\le q\le d-1$, and such that $\sigma_2-\sigma_1\le 1$, let $m\ge 1$ be the minimal integer such that $\left(\frac{ke}{m+1}\right)^{m+1}\le \frac{\gamma}{2}$. Then there exists $C=C(d)$ such that for any integer $n\ge C^m$ there exist $(y_j)_{j=1}^n\subseteq(\sigma_1,\sigma_2)$ satisfying
\begin{equation}\label{one_dim_quadrature_for_sine_estimate}
\left|\frac{1}{\int_{\sigma_1}^{\sigma_2} \sin^q(\theta)d\theta}\int_{\sigma_1}^{\sigma_2} \sin^{k_1+q}(\theta)\cos^{k_2}(\theta)d\theta - \frac{1}{n}\sum_{j=1}^n \sin^{k_1}(y_j)\cos^{k_2}(y_j)\right|\le\gamma
\end{equation}
for all integers $k_1, k_2\ge 0$ such that $k_1+k_2\le k$.
\end{lemma}


We first show how to use the lemma, then give its proof.

\begin{proof}[Proof of Theorem \ref{special_sphere_cubature}]
Denote $I_0:= J$. Using the lemma, for each $0\le q\le d-1$, let $(y_{q,i,j})_{j=1}^n$ be the $(y_j)$ satisfying $\eqref{one_dim_quadrature_for_sine_estimate}$ for the given $k$ and for $\gamma:=\frac{\delta}{d2^k}$. Let $(x_{i,j})_{j=1}^{n^d}$ be the Cartesian product $(y_{0,i,j})_{j=1}^n\times\cdots\times (y_{d-1,i,j})_{j=1}^n$. Finally let $(z_{i,j})_{j=1}^{n^d}$ be defined by $z_{i,j}:=T(x_{i,j})$. Note that for any $h$ as in \eqref{polynomial_transformation_to_sphere}, we have by \eqref{polynomial_transformation_to_sphere}, \eqref{one_dim_quadrature_for_sine_estimate} and using that $|\sin(\theta)|\le 1$ and $|\cos(\theta)|\le 1$ that
\begin{equation*}
\left|\frac{1}{\sigma_d(E)}\int_E h(z)d\sigma_d(z)-\frac{1}{n^d}\sum_{j=1}^{n^d} h(z_{i,j})\right| \le d\gamma.
\end{equation*}
To finish the proof of the theorem, it remains to show that the $z_{i,j}$ provide an approximate Chebyshev-type quadrature also for $g(z)$ of the form $g(z)=(z-w)^\alpha$ for $w\in \S^d$ and a multi-index $\alpha$ with $|\alpha|\le k$. Fix such a $g$. For a multi-index $\beta\in(\N\cup\{0\})^{d+1}$ we write $\beta\preceq \alpha$ if $\beta_q\le \alpha_q$ for all $q$. Then
\begin{multline*}
\left|\frac{1}{\sigma_d(E)}\int_E g(z)d\sigma_d(z)-\frac{1}{n^d}\sum_{j=1}^{n^d} g(z_{i,j})\right|=\\
=\left|\frac{1}{\sigma_d(E)}\int_E \sum_{\beta\preceq\alpha} z^\beta(-w)^{\alpha-\beta}\prod_{q=1}^{d+1} \binom{\alpha_q}{\beta_q}d\sigma_d(z)-\frac{1}{n^d}\sum_{j=1}^{n^d} \sum_{\beta\preceq\alpha} z_{i,j}^\beta(-w)^{\alpha-\beta}\prod_{q=1}^{d+1} \binom{\alpha_q}{\beta_q}\right|\le\\
\le\sum_{\beta\preceq\alpha} d\gamma\prod_{q=1}^{d+1} \binom{\alpha_q}{\beta_q}= d\gamma 2^{|\alpha|}\le d\gamma 2^k=\delta
\end{multline*}
which completes the proof of the theorem.
\end{proof}

\begin{proof}[Proof of Lemma~\ref{one_dim_quadrature_for_sine_lemma}]
Fix $\gamma>0$, integers $k\ge 1, 0\le q\le d-1$ and $k_1,k_2\ge 0$ such that $k_1+k_2\le k$. Fix also $0\le \sigma_1 < \sigma_2\le \tau_q$ satisfying $\sigma_2 - \sigma_1\le 1$. Let $f(\theta):=\sin^{k_1}(\theta)\cos^{k_2}(\theta)$ and write
\begin{equation*}
f(\theta)=\sum_{i=0}^m a_i(\theta-\sigma_1)^i + r_{m}(\theta)(\theta-\sigma_1)^{m+1}
\end{equation*}
the Taylor expansion with remainder term of $f(\theta)$ up to degree $m$. Recall that $r_m(\theta)=\frac{f^{(m+1)}(\tilde{\theta})}{(m+1)!}$ where for $\theta\ge \sigma_1$, $\tilde{\theta}$ is some number in $(\sigma_1, \theta)$. By the Cauchy estimates, we have for any $\rho>0$, 
\begin{equation*}
|r_m(\theta)|\le \rho^{-(m+1)}\max_{\substack{|z|=\rho\\\tilde{\theta}\in(\sigma_1,\theta)}}|f(z+\tilde{\theta})|\le\rho^{-(m+1)}\max_{\substack{|z|=\rho\\\tilde{\theta}\in(\sigma_1,\theta)}}|e^{i(k_1+k_2)(z+\tilde{\theta})}|\le \rho^{-(m+1)}e^{k\rho}.
\end{equation*}
Choosing $\rho=\frac{m+1}{k}$ we obtain $|r_m(\theta)|\le \left(\frac{ke}{m+1}\right)^{m+1}$. We thus choose $m\ge 1$ to be minimal such that $\left(\frac{ke}{m+1}\right)^{m+1}\le \frac{\gamma}{2}$. If we now find $(y_j)_{j=1}^n$ satisfying
\begin{equation}\label{quadrature_for_sine_weight}
\frac{1}{n}\sum_{j=1}^n y_j^s = \frac{1}{\int_{\sigma_1}^{\sigma_2} \sin^q(\theta)d\theta}\int_{\sigma_1}^{\sigma_2} \theta^s \sin^q(\theta)d\theta
\end{equation}
for all integer $0\le s\le m$, then these $(y_j)$ will satisfy the requirements of the lemma since
\begin{multline*}
\left|\frac{1}{\int_{\sigma_1}^{\sigma_2} \sin^q(\theta)d\theta}\int_{\sigma_1}^{\sigma_2} f(\theta)\sin^q(\theta)d\theta - \frac{1}{n}\sum_{j=1}^n f(y_j)\right|\le\\
\le \left|\frac{1}{\int_{\sigma_1}^{\sigma_2} \sin^q(\theta)d\theta}\int_{\sigma_1}^{\sigma_2} r_m(\theta)(\theta-\sigma_1)^{m+1}\sin^q(\theta)d\theta-\frac{1}{n}\sum_{j=1}^n r_m(y_j)(y_j-\sigma_1)^{m+1}\right| \le \gamma,
\end{multline*}
where we used $|\sigma_2-\sigma_1|\le 1$ and $\sin^q(\theta)\ge 0$ for $\theta\in[0,\tau_q]$ in the last inequality.
To find $(y_j)$'s satisfying \eqref{quadrature_for_sine_weight}, first scale the problem from $[\sigma_1,\sigma_2]$ to the $[0,1]$ interval. Theorem~\ref{abs_cont_upper_bound_thm} then shows that $(y_j)_{j=1}^n$ exist for any integer $n\ge 75e^4 mM\left(12eM\right)^{(m-1)}$ where $M=\frac{\sigma_2-\sigma_1}{\int_{\sigma_1}^{\sigma_2} \sin^q(\theta)d\theta}\max_{\theta\in[\sigma_1,\sigma_2]}\sin^q(\theta)$. The lemma follows since $M\le C(q)$ independently of $\sigma_1$ and $\sigma_2$.
%
\end{proof}


\end{subsubsection}
\begin{subsubsection}{Cylinder Construction}
We prove Theorem~\ref{special_cubature_on_cyl_thm} in the special case $W=1$. The general case follows from this as follows. If we want the general case with parameters $k, L',W',\tau'$ and $\delta'$ we can take the $W=1$ construction with the same $k$ and parameters $L:=\frac{L'}{W'}, W=1, \tau:=\frac{\tau'}{W'}, \delta:=\frac{\delta'}{(W')^k}$ and rescale its result by a factor of $W'$.
 

\begin{proof}[Proof of Theorem~\ref{special_cubature_on_cyl_thm} for $W=1$]
We first use Theorem~\ref{special_sphere_cubature} with its input parameters $d, k, \tau, \delta$ taken to be, in terms of the input parameters to Theorem~\ref{special_cubature_on_cyl_thm}, $d-2, k, \tau, \frac{\delta}{(4L)^k}$, respectively. We thus obtain sets $E_1,\ldots, E_{K'}\subseteq\S^{d-2}$ and, for each $1\le i\le K'$, $(z_{i,j})_{j=1}^N\subseteq E_i$, satisfying the assertions of Theorem~\ref{special_sphere_cubature}.
Next, for each $1\le i\le K'$, we define intervals $(I_{i,q})_{q=1}^{m_i}\subseteq[-\frac{3}{2}L,\frac{3}{2}L]$ by the following procedure: $I_{i,q}:= [a_{i,q-1}, a_{i,q}]$ with $a_{i,0}:=-\frac{3}{2}L$, with subsequent $a_{i,q}$'s defined by the rule 
\begin{equation}\label{I_i_q_requirement}
\int_{[a_{i,q-1},a_{i,q}]}v(x)dx=\frac{\tau^{d-1}}{\sigma_{d-2}(E_i)}, 
\end{equation}
and with the integer $m_i$ set to be the maximal one for which $[a_{i,m_i-1},a_{i,m_i}]\subseteq[-\frac{3}{2}L,\frac{3}{2}L]$. Here we recall that, by Property (1) of Theorem~\ref{special_sphere_cubature}, we have that $\frac{\tau^{d-1}}{\sigma_{d-2}(E_i)}\le C'(d)\tau$ for some $C'(d)>0$ and, thus, using that $1\le v(x)\le 2$, $\tau<1$ and our assumption that $L>C(d)$, we can, and do, take $C(d)$ sufficiently large to ensure that $m_i$ is well-defined and $a_{i,m_i}>L$ for all $i$. Note also that the $(I_{i,q})$ satisfy $\diam(I_{i,q})\le \tilde{C}(d)\tau$ for some $\tilde{C}(d)>0$ and all $i$ and $q$. Finally, we take $K:=\sum_{i=1}^{K'} m_i$ and the required sets $(D_j)_{j=1}^{K}$ to be all sets of the form $I_{i,q}\times E_i$ for $1\le i\le K'$ and $1\le q\le m_i$.

We continue by establishing properties (I) and (II). By the fact that the $(E_i)$ are disjoint up to surface measure 0, and our construction, it follows that the $(D_j)$ are disjoint up to $\nu_{2L,1}$-measure 0. Furthermore, as remarked above, we have $a_{i,m_i}>L$ for all $i$ and, thus, $\nu_{2L,1}\left(P_{L,1}\setminus\left(\cup_{j=1}^{K} D_j\right)\right)=0$, proving (I).  In addition, as noted above, $\diam(I_{i,q})\le \tilde{C}(d)\tau$ for all $i$ and $q$ and thus, by the diameter bound on the $(E_i)$ and our construction, it follows that $\diam(D_j)\le \bar{C}(d)\tau$ for some $\bar{C}(d)>0$ and all $j$. Relation \eqref{I_i_q_requirement} ensures that $\nu_{2L,1}(D_j)=\tau^{d-1}$ for all $j$ and this implies the bound on $K$ by a volume estimate. Thus, (II) is proved.


It remains to establish (III). Fix $1\le i\le K'$, $1\le q\le m_i$ and, for brevity, denote $D:=I_{i,q}\times E_i$. We start by constructing the points $(w_{D,j})_{j=1}^n$. We apply Theorem~\ref{abs_cont_upper_bound_thm} to the interval $I_{i,q}$ with the weight $v$ restricted to that interval (by rescaling the interval to $[0,1]$ and renormalizing the measure to be a probability measure) and obtain $(x_{i,q,j})_{j=1}^{n_0}\subseteq [a_{i,q-1},a_{i,q}]$ satisfying
\begin{equation}\label{quadrature_on_interval_of_cyl}
\frac{1}{n_0}\sum_{j=1}^{n_0} x_{i,q,j}^r=\frac{1}{\int_{a_{i,q-1}}^{a_{i,q}} v(x)dx}\int_{a_{i,q-1}}^{a_{i,q}} x^r v(x)dx
\end{equation}
for any integer $0\le r\le k$. Furthermore, $n_0$ may be any integer such that $n_0\ge75e^4 kM\left(12eM\right)^{(k-1)}$ where $M:=\frac{a_{i,q}-a_{i,q-1}}{\int_{a_{i,q-1}}^{a_{i,q}} v(x)dx}\max_{x\in[a_{i,q-1},a_{i,q}]}v(x)\le 2$. Hence, there exists $C>0$, independent of all other parameters, such that $n_0$ may be any integer satisfying $n_0\ge C^k$. Finally, we take the points $(w_{D,j})_{j=1}^n$ to be the Cartesian product of $(x_{i,q,j})_{j=1}^{n_0}$ and $(z_{i,j})_{j=1}^N$ (of Theorem~\ref{special_sphere_cubature}). This implies $n:=n_0 N$, where $N=n_1^{d-2}$ and $n_1$ can be any integer satisfying $n_1\ge C_1(d)^{m_0(d-2,k,\delta/(4L)^k)}$ for some $C_1(d)>0$. Note that $m_0(d-2,k,\delta/(4L)^k)\ge k$ (using that $L>C(d)$ and taking $C(d)>1$), implying that $n$ can be any integer of the form $n_2^{d-1}$ for $n_2\ge C_2(d)^{m_0(d-2,k,\delta/(4L)^k)}$, for some $C_2(d)>0$, as required. 

It remains only to verify the approximate quadrature condition of (III). Fix $h:\R^d\to\R$, $h(w):=(w-y)^{\alpha}$ for $y\in P_{2L,1}$ and a multi-index $\alpha$ with $|\alpha|\le k$. Let $\tilde{h}:\R^{d-1}\to\R$ be defined by $\tilde{h}(\tilde{w}):=(\tilde{w}-\tilde{y})^{\tilde{\alpha}}$, where $\tilde{y}\in R^{d-1}$ is defined by $\tilde{y}_i:=y_{i+1}$ (so that $\tilde{y}\in\S^{d-2})$ and $\tilde{\alpha}\in(\N\cup\{0\})^{d-1}$ is defined by $\tilde{\alpha}_i:=\alpha_{i+1}$. Using the product structure of the $(w_{D,j})$, $D$, $\nu_{2L,1}$ and $h$, and \eqref{quadrature_on_interval_of_cyl} and \eqref{approx_quadrature_on_sphere}, we obtain
\begin{equation*}
\begin{split}
&\left|\frac{1}{n}\sum_{j=1}^n h(w_{D,j})-\frac{1}{\nu_{2L,1}(D)}\int_D h(w)d\nu_{2L,1}(w)\right|=\\
&=\left|\frac{1}{\int_{a_{i,q-1}}^{a_{i,q}} v(w_1)dw_1}\int_{a_{i,q-1}}^{a_{i,q}} (w_1-y_1)^{\alpha_1} v(w_1)dw_1\right|\cdot\\
&\qquad\qquad\cdot\left|\frac{1}{N}\sum_{j=1}^N \tilde{h}(z_{i,j})-\frac{1}{\sigma_{d-2}(E_i)}\int_{E_i} \tilde{h}(\tilde{w})d\sigma_{d-2}(\tilde{w})\right|\le\\
&\le\max_{w,y\in P_{2L,1}} |w_1-y_1|^{\alpha_1} \cdot\frac{\delta}{(4L)^k}\le \delta,
\end{split}
\end{equation*}
where we used that $L>C(d)$ and took $C(d)>1$ in the last inequality.\qedhere


\end{proof}
\end{subsubsection}
\end{subsection}
\end{section}
\begin{section}{Open questions}\label{open_questions_section}
\begin{enumerate}
\item What is the best possible upper bound on $n_\sigma^0(k)$ and $n_\sigma(k)$ using only the information contained in $R_\sigma$ or using other \emph{simple} properties of $\sigma$? Can the conclusion of Theorem~\ref{one_dim_quadrature_thm} be improved? In our lower bounds section, we showed that Theorem~\ref{one_dim_quadrature_thm} is sharp, up to constants, for measures with bounded density. Is this also the case for other measures?

\item For which measures $\sigma$ does $n_\sigma^0(k)$ or $n_\sigma(k)$ grow only polynomially with $k$? Can such behavior be deduced using only $\emph{simple}$ properties of $\sigma$ (i.e., without knowing the orthogonal polynomials of $\sigma$)? Is it always the case for measures on a finite interval, with a density which is bounded away from 0 and infinity?
\item Salkauskas \cite{S75} puts a probability measure on the set of length $k$ moment vectors (normalized volume measure) from which he deduces that the proportion of those vectors for which a Chebyshev quadrature (i.e., a Chebyshev-type quadrature with $n$ nodes and algebraic degree at least $n$) exists is exponentially small in $n$. Can Salkauskas' result be extended to give the typical degree of a $n$ node Chebyshev-type quadrature? Is this typical degree a power of $n$ or logarithmic in $n$?
\item Obtain quantitative theorems for \emph{random} Chebyshev-type quadratures. Do they help to show existence of Chebyshev-type quadratures? In particular, for the case of the uniform measure on the cube, what is the order of magnitude of $N_0$ in Theorem~\ref{positive_density_thm}? (see also remarks preceding the theorem)
\item Theorem~\ref{one_dim_quadrature_thm} and its extensions give upper bounds for $n_\sigma(k)$ for measures $\sigma$ supported on a finite interval. Can similar theorems be written for measures in higher dimensions? or for measures with unbounded support?
\item Theorems~\ref{one_dim_quadrature_thm} and \ref{large_atom_upper_bound_theorem} give upper bounds for $n_\sigma(k)$ in the cases when $\sigma$ has at least $k+4$ atoms, or a non-atomic part. However, we know from Theorem~\ref{CTQ_existence_theorem} that Chebyshev-type quadratures exist once we have roughly $\frac{k}{2}$ atoms. Can we bound $n_\sigma(k)$ using simple properties of $\sigma$ for measures having between $\frac{k}{2}$ and $k+4$ atoms?
\end{enumerate}
\end{section}


\end{document}